\documentclass[11pt]{article}
\usepackage{etex}
\usepackage{mathtools}
\usepackage{enumitem}

\usepackage{amsmath,amsthm}
\usepackage{csquotes}
\usepackage{todonotes}

\makeatletter
\let\orgdescriptionlabel\descriptionlabel
\renewcommand*{\descriptionlabel}[1]{
	\let\orglabel\label
	\let\label\@gobble
	\phantomsection
	\edef\@currentlabel{#1} 	
	\let\label\orglabel
	\orgdescriptionlabel{#1}
}
 											
\def\th@plain{
	\thm@notefont{}
	\itshape
}
\def\th@definition{
	\thm@notefont{}
	\normalfont
}

\g@addto@macro\th@remark{\thm@headpunct{}}
\g@addto@macro\th@definition{\thm@headpunct{}}
\g@addto@macro\th@plain{\thm@headpunct{}}
\makeatother

\usepackage{amssymb}

\usepackage{graphicx}

\usepackage{subfig}
\usepackage[final]{showkeys}

\usepackage{etoolbox}

\usepackage{wasysym}

\usepackage{mathrsfs}

\usepackage{mathpazo}

\usepackage[titletoc]{appendix}
\usepackage[doc]{optional}

\usepackage{soul}

\usepackage{colortbl,booktabs,sectsty,multirow}
\usepackage{xcolor}

\usepackage{cancel}

\usepackage{empheq}

\definecolor{myblue}{rgb}{.8, .8, 1}
  \newcommand*\mybluebox[1]{
    \colorbox{myblue}{\hspace{1em}#1\hspace{1em}}}

\usepackage[obeyspaces,hyphens,spaces]{url}

\usepackage{hyperref}
\hypersetup{
    colorlinks=true,
    linkcolor=blue,
    citecolor=blue,
    filecolor=magenta,
    urlcolor=cyan
}

\usepackage[
  open,
  openlevel=2,
  atend,
  numbered
]{bookmark}

\usepackage[capitalize, nameinlink, noabbrev]{cleveref}
\crefname{equation}{}{}
\crefname{chapter}{Chapter}{Chapters}
\crefname{item}{item}{items}
\crefname{figure}{Figure}{Figures}
\crefname{theorem}{Theorem}{Theorems}
\crefname{lemma}{Lemma}{Lemmas}
\crefname{proposition}{Proposition}{Propositions}
\crefname{corollary}{Corollary}{Corollarys}
\crefname{definition}{Definition}{Definitions}
\crefname{fact}{Fact}{Facts}
\crefname{example}{Example}{Examples}
\crefname{algorithm}{Algorithm}{Algorithms}
\crefname{remark}{Remark}{Remarks}
\crefname{note}{Note}{Notes}
\crefname{notation}{Notation}{Notations}
\crefname{case}{Case}{Cases}
\crefname{exercise}{Exercise}{Exercises}
\crefname{question}{Question}{Questions}
\crefname{claim}{Claim}{Claims}
\crefname{enumi}{}{}

\usepackage[top= 2cm, bottom = 2 cm, left = 2.2 cm, right= 2.2 cm]{geometry}

\usepackage{float}

\usepackage{pgf}

\usepackage{array}
\usepackage{tabu}

\allowdisplaybreaks

\numberwithin{equation}{section}

\theoremstyle{plain}
\newtheorem{theorem}{Theorem}[section]

\newtheorem{corollary}[theorem]{Corollary}
\newtheorem{fact}[theorem]{Fact}
\newtheorem{lemma}[theorem]{Lemma}
\newtheorem{proposition}[theorem]{Proposition}

\theoremstyle{definition}
\newtheorem{definition}[theorem]{Definition}
\newtheorem{example}[theorem]{Example}

\newtheorem{remark}[theorem]{Remark}

\newcommand{\conv}{\ensuremath{\operatorname{conv}}}
\newcommand{\aff}{\ensuremath{\operatorname{aff} \,}}

\newcommand{\Fix}{\ensuremath{\operatorname{Fix}}}
\newcommand{\Id}{\ensuremath{\operatorname{Id}}}

\newcommand{\Pro}{\ensuremath{\operatorname{P}}}
\newcommand{\R}{\ensuremath{\operatorname{R}}}
\newcommand{\I}{\ensuremath{\operatorname{I}}}
\newcommand{\J}{\ensuremath{\operatorname{J}}}

\newcommand{\Range}{\ensuremath{\operatorname{ran}}}

\newcommand{\CCO}[1]{CC{#1}}

\newcommand{\CC}[1]{CC_{#1}}

\providecommand{\norm}[1]{\lVert#1\rVert}
\providecommand{\Norm}[1]{{\Big\lVert}#1{\Big\rVert}}
\providecommand{\bignorm}[1]{{\big\lVert}#1{\big\rVert}}

\providecommand{\innp}[1]{\langle#1\rangle}
\providecommand{\Innp}[1]{\Big\langle#1\Big\rangle}

\newcommand\scalemath[2]{\scalebox{#1}{\mbox{\ensuremath{\displaystyle #2}}}}

\begin{document}

\title{ \sffamily  On the linear convergence of circumcentered isometry methods}

\author{
         Heinz H.\ Bauschke\thanks{
                 Mathematics, University of British Columbia, Kelowna, B.C.\ V1V~1V7, Canada.
                 E-mail: \href{mailto:heinz.bauschke@ubc.ca}{\texttt{heinz.bauschke@ubc.ca}}.},~
         Hui\ Ouyang\thanks{
                 Mathematics, University of British Columbia, Kelowna, B.C.\ V1V~1V7, Canada.
                 E-mail: \href{mailto:hui.ouyang@alumni.ubc.ca}{\texttt{hui.ouyang@alumni.ubc.ca}}.},~
         and Xianfu\ Wang\thanks{
                 Mathematics, University of British Columbia, Kelowna, B.C.\ V1V~1V7, Canada.
                 E-mail: \href{mailto:shawn.wang@ubc.ca}{\texttt{shawn.wang@ubc.ca}}.}
                 }

\date{April 12, 2020}

\maketitle

\begin{abstract}
\noindent
The circumcentered Douglas--Rachford method (C--DRM), introduced by Behling, Bello Cruz and Santos, iterates by taking the circumcenter of associated successive reflections. It is an acceleration of the well-known Douglas-Rachford method (DRM) for finding the best approximation  onto the intersection of finitely many affine subspaces.  Inspired by the C--DRM, we introduced the more flexible circumcentered reflection method (CRM) and circumcentered isometry method (CIM). The CIM essentially chooses the closest point to the solution among all of the points in an associated affine hull as its iterate and is a generalization of the CRM.
The circumcentered--reflection method introduced by Behling, Bello Cruz and Santos to generalize the C--DRM is a special class of our CRM.

We consider the CIM induced by a set of finitely many isometries for finding the best approximation onto the intersection of fixed point sets of the isometries which turns out to be an intersection of finitely many affine subspaces.
We extend our previous linear convergence results on CRMs in finite-dimensional spaces from reflections to isometries.
In order to better accelerate the symmetric method of alternating projections (MAP), the accelerated symmetric MAP first applies another operator to the initial point. (Similarly, to accelerate the DRM, the C--DRM first applies another operator to the initial point as well.) Motivated by these facts, we show results on the linear convergence of CIMs in Hilbert spaces with first applying another operator to the initial point. In particular, under some restrictions, our results imply that some CRMs attain the known linear convergence rate of the accelerated symmetric MAP in Hilbert spaces. 
We also exhibit a class of CRMs converging to the best approximation in Hilbert spaces with a convergence rate no worse than the sharp convergence rate of MAP.
The fact that some CRMs attain the linear convergence rate of MAP or accelerated symmetric MAP is entirely new. 
\end{abstract}

{\small
\noindent
{\bfseries 2020 Mathematics Subject Classification:}
{Primary 41A50, 47H30, 65B99;
Secondary 46B04, 90C25.}

\noindent{\bfseries Keywords:}
Isometry, projector, reflector, Friedrichs angle, best approximation problem, linear convergence, circumcentered isometry method, circumcentered reflection method,  method of alternating projections, accelerated symmetric method of alternating projections.
}

\section{Introduction} \label{sec:Introduction}

Throughout this paper, we assume that
\begin{empheq}[box = \mybluebox]{equation*}
\text{$\mathcal{H}$ is a real Hilbert space},
\end{empheq}
with inner product $\innp{\cdot,\cdot}$ and induced norm $\|\cdot\|$.
Denote by $\mathcal{P}(\mathcal{H})$ the set of nonempty subsets of $\mathcal{H}$ containing finitely many elements.  The circumcenter operator  $\CCO{} \colon \mathcal{P}(\mathcal{H}) \to \mathcal{H} \cup \{ \varnothing \}$
maps every $K \in \mathcal{P}(\mathcal{H})$ to the  \emph{circumcenter $\CCO{(K)}$ of $K$}, where
$\CCO{(K)}$  is either the empty set or the unique point $\CCO{(K)}$ such that $\CCO{(K)} \in \aff (K)$ and $\CCO{(K)}$ is equidistant  from all points  in $K$ (see \cite[Proposition~3.3]{BOyW2018}).

Throughout the paper, $\mathbb{N}=\{0,1,2,\ldots\}$, $m \in \mathbb{N} \smallsetminus \{0\}$ and
\begin{empheq}[box=\mybluebox]{equation*}
\big(\forall i \in \{1, \ldots, m\} \big) \quad T_{i} : \mathcal{H} \rightarrow \mathcal{H} ~\text{is affine isometry} \quad \text{with} \quad \bigcap^{m}_{j=1} \Fix T_{j} \neq \varnothing.
\end{empheq}
Unless stated otherwise, we set
\begin{empheq}[box = \mybluebox]{equation*}
\mathcal{S} :=\{ T_{1}, \ldots, T_{m-1}, T_{m} \}. 
\end{empheq}
The associated set-valued operator $\mathcal{S}: \mathcal{H} \rightarrow \mathcal{P}(\mathcal{H})$ is defined by
\begin{empheq}[box = \mybluebox]{equation*}
(\forall x \in \mathcal{H}) \quad \mathcal{S}(x) :=\{ T_{1}x, \ldots, T_{m-1}x, T_{m}x\}.
\end{empheq}
The \emph{circumcenter mapping $\CC{\mathcal{S}}$ induced by $\mathcal{S}$} is defined by the composition of  $\CCO{}$ and $\mathcal{S}$, that is $(\forall x \in \mathcal{H})$ $\CC{\mathcal{S}}(x) :=\CCO{} \left( \mathcal{S}(x)  \right)$. 
Inspired by the circumcentered Douglas-Rachford method (C--DRM)  introduced by Behling, Bello Cruz and Santos \cite{BCS2017}, we proved in  \cite[Theorem~3.3]{BOyW2019Isometry} that the $\CC{\mathcal{S}}$ is \emph{proper}, i.e., $ (\forall x \in \mathcal{H})$, $\CC{\mathcal{S}}(x) \in \mathcal{H}$. Hence,  we are able to define the \emph{circumcenter method induced by  $\mathcal{S}$} as
\begin{align*} 
x_{k} :=\CC{\mathcal{S}}(x_{k-1})=\CC{\mathcal{S}}^{k}(x_{0}), ~\mbox{where}~x_{0} \in \mathcal{H} ~\text{and}~ k=1,2,\ldots.
\end{align*}

Since every element of  $\mathcal{S}$ is isometry, we say that the circumcenter method induced by the $\mathcal{S}$ is the \emph{circumcentered isometry method} (CIM). Since reflectors associated with affine subspaces are isometries, we call  the circumcenter method induced by a set of reflectors the \emph{circumcentered reflection method} (CRM). 

\emph{Our goal in this paper is to study the linear convergence of CIMs in Hilbert spaces for
finding the best approximation $\Pro_{\cap_{i=1}^{m} \Fix T_{i} }x$  onto the intersection of finitely many affine subspaces,
where $x \in \mathcal{H}$ is an arbitrary but fixed point. 
In particular, given affine subspaces $U_{1}, U_{2}, \ldots, U_{m}$ with $\cap^{m}_{i=1} U_{i} \neq \varnothing$,  finding the best approximation $\Pro_{\cap^{m}_{i=1} U_{i}}x$  
is covered by our work.}

The main results in this paper are the following.
\begin{itemize}
	\item[\textbf{R1:}]  \Cref {theo:CCS:LineaConve:F1Fn,theo:CCS:LineaConve:F1t} extend the \cite[Propositions~5.15 and 5.10]{BOyW2019Isometry} respectively from reflections to isometries and establish the linear convergence of CIMs for finding the best approximation onto the intersection of the fixed point sets of finitely many isometries in finite-dimensional spaces.  Moreover, \cite[Theorem~3.3]{BCS2018} is a special instance of \cref{theo:CCS:LineaConve:F1t}. 
	\item[\textbf{R2:}] \cref{theo:LineaConvIneq:TF} provides two sufficient conditions for the linear convergence of CIMs in Hilbert spaces with first applying another operator on the initial point. The applications of \cref{theo:LineaConvIneq:TF} can be found in \cref{theo:CCS:Accel:AT:Tx},  \cite[Proposition~5.19]{BOyW2019Isometry} and \cite[Theorem~1]{BCS2017}.
	\item[\textbf{R3:}]  \Cref{thm:MAP:LC,theo:CCS:Accel:AT:x,theo:CCS:Accel:AT:Tx} present sufficient conditions for the linear convergence of CRMs  for finding the best approximation onto the intersection of finitely many closed linear subspaces in Hilbert spaces, by using the linear convergence of MAP and accelerated symmetric MAP. 
\end{itemize}

In fact, we generalize all of results on  the linear convergence of CRMs shown in \cite{BCS2017}, \cite{BCS2018} and  \cite{BOyW2019Isometry} from reflections to isometries. We prove in \cref{theo:PARALLIsometryLinear}  that the linear convergence of any general CIM is equivalent to the linear convergence of the CIM induced by a corresponding set of linear isometries. Hence, to study the linear convergence of CIM, we are free in our proofs to assume that all of the related isometries are linear.  We also prove in \cref{theo:TselfAdjoinIsome:ProFixT}\cref{theo:TselfAdjoinIsome:ProFixT:RT} that given a linear isometry $T$, $T$ is reflector if and only if $T$ is self-adjoint.  In fact, the linear isometries on $\mathbb{R}^{n}$ are precisely orthogonal matrices. But orthogonal matrices are in general not symmetric. Hence, our generalizations are indeed less restrictive. 

In \cite{BDHP2003}, Bauschke, Deutsch, Hundal and Park studied the acceleration scheme for  linear nonexpansive operators which was considered by  Gubin, Polyak, and Raik \cite{GPR1967} and by  Gearhart and Koshy \cite{GK1987}. It was  proved that the  acceleration scheme  for (symmetric) MAP is indeed faster than the (symmetric)  MAP. 
Note that \cref{exam: lc:cw2}, which is a corollary of \cref{thm:MAP:LC}, states that the convergence rate of some CRMs is no worse than the sharp convergence rate of MAP in Hilbert spaces. Moreover, \Cref {theo:CCS:Accel:AT:x,theo:CCS:Accel:AT:Tx} illustrate that some CRMs attain the known linear convergence rate of the accelerated symmetric MAP in Hilbert spaces.  In fact, in \cite[Section~6]{BOyW2019Isometry} we showed numerically the  outstanding performance of some instances of those CRMs without analytical proof by comparing four CRMs with MAP and   DRM. Now, \Cref{thm:MAP:LC,theo:CCS:Accel:AT:x} provide theoretical support for the results presented by the numerical experiments in \cite[Section~6]{BOyW2019Isometry}. 

For the readers who are interested in  CRMs for general convex or nonconvex feasibility problems, we recommend  \cite{BCS2020ConvexFeasibility}, \cite{DHL2019} and \cite{SBLL2020}.

The paper is organized as follows.  In \Cref{sec:AuxiliaryResults,sec:Isometries}, we collect  various auxiliary results to facilitate the proofs in the sequel.  Some results are interesting on their own (see \cref{prop:operatornorm:largestgeigenvector}, and \Cref{theo:TselfAdjoinIsome:ProFixT,theo:LinearSelfAdIsometry:Fix}).
Some properties of CRMs shown in \cite{BOyW2019Isometry} are generalized to CIMs in \cref{sec:CircumcenteredIsometryMethods}. 
\cref{section:LinearConverg}  focuses on the linear convergence of CIMs for finding the best approximation onto intersections of fixed point sets of finitely many affine isometries. More precisely, in \cref{section:LinearConverg},
the linear convergence of CIMs in $\mathbb{R}^{n}$ is presented, and  two sufficient conditions for the linear convergence of CIMs in Hilbert spaces with first applying another operator to the initial point are provided.
In \cref{sec:LinearConverCRMHilert}, we use the linear convergence of MAP to deduce sufficient conditions for the  linear convergence of CRMs in Hilbert spaces. We  also provide examples of CRMs with convergence rate no worse than the sharp convergence rate of MAP.  In addition, we  prove that some CRMs attain the known convergence rate of the accelerated symmetric MAP.

We now turn to the notation used in this paper. Let $C$ be a nonempty subset
of $\mathcal{H}$.  
$C$ is an \emph{affine subspace} of
$\mathcal{H}$ if $C \neq \varnothing$ and $(\forall \rho\in\mathbb{R})$ $\rho
C + (1-\rho)C = C$. The smallest affine subspace of $\mathcal{H}$ containing $C$ is
the denoted by $\aff C$ and called the \emph{affine hull} of $C$. The \emph{orthogonal
complement of $C$} is the set $ C^{\perp} :=\{x \in \mathcal{H}~|~ \innp{x,y}=0
~\text{for all}~y \in C\}$.
The \emph{best approximation operator} (or \emph{projector}) onto $C$ is denoted by
$\Pro_{C}$. $\R_{C} :=2 \Pro_{C} -\Id$ is the \emph{reflector associated with $C$}.

Let
$T: \mathcal{H} \rightarrow \mathcal{H}$ be an operator. Let $\ker T := \{x \in
\mathcal{H} ~|~ Tx =0 \}$ be the \emph{kernel of $T$}. The \emph{set of fixed
points of the operator $T$} is denoted by $\Fix T$, i.e., $\Fix T := \{x \in
\mathcal{H} ~|~ Tx=x\}$. The \emph{range of $T$} is defined as $\Range T :=\{Tx ~:~ x \in \mathcal{H} \}$;  moreover, $\overline{\Range} \, T$ is the closure of $\Range T$. Denote by $\mathcal{B} (\mathcal{H}) := \{ T: \mathcal{H} \rightarrow \mathcal{H} ~:~ T ~\text{is bounded and linear} \}$. For every $T \in \mathcal{B} (\mathcal{H})$, the \emph{operator norm $\norm{T}$ of $T$} is defined by $\norm{T} := \sup_{\norm{x} \leq  1} \norm{Tx}$. 
Let $m, n$ be in $\mathbb{N} \smallsetminus \{0\}$ and let $A \in \mathbb{R}^{n \times m}$.   
The \emph{matrix $2$-norm induced by the Euclidean vector norm} is $\norm{A}_{2} := \max_{\norm{x}_{2} \leq 1} \norm{Ax}_{2}$. 
For other notation not explicitly defined here, we refer the reader to \cite{BC2017}.

\section{Auxiliary results} \label{sec:AuxiliaryResults}

To facilitate the proofs in our main results in the sequel, we collect and prove some useful results in this section. 
\subsection*{Projections and Friedrichs angles}

\begin{fact} {\rm \cite[Proposition~3.19]{BC2017}} \label{fac:SetChangeProje}
	Let $C$ be a nonempty closed  convex subset of $\mathcal{H}$ and let $x \in \mathcal{H}$. Set $D :=z+C$, where $z \in \mathcal{H}$. Then $\Pro_{D}x=z+\Pro_{C}(x-z)$.
\end{fact}

\begin{fact} {\rm \cite[Theorem~4.9]{D2012}} \label{fact:CharaProjectionLineaSpace}
	Let $M$ be a linear subspace in $\mathcal{H}$, $x \in \mathcal{H}$, and $p \in M$. Then $p = \Pro_{M}x$ if and only if $x - p \in M^{\perp}$; that is,
	$(\forall y \in M)$	 $\innp{x-p, y} = 0$. 
\end{fact}

\begin{fact} {\rm \cite[Theorems~3.5 and 5.5]{D2012} }  \label{fact:ProjectorInnerPRod}
	Let $C$ be a nonempty closed  convex set of $\mathcal{H}$. Then the following assertions hold:
	\begin{enumerate}
		\item  \label{fact:ProjectorInnerPRod:Idempotent} $\Pro_{C}$ is idempotent: $\Pro_{C}^{2} =\Pro_{C}$.
		\item 	\label{fact:ProjectorInnerPRod:FirmlyNonexp}  $\Pro_{C}$ is firmly nonexpansive:  $(\forall x \in \mathcal{H})$ $(\forall y \in \mathcal{H})$ $\innp{x-y, \Pro_{C}x - \Pro_{C}y} \geq \norm{\Pro_{C}x -\Pro_{C}y}^{2}$.
		\item  \label{fact:ProjectorInnerPRod:Nonnegative} $\Pro_{C}$ is  monotone: $(\forall x \in \mathcal{H})$ $(\forall y \in \mathcal{H})$ $\innp{x-y, \Pro_{C}x - \Pro_{C}y} \geq 0$.
	\end{enumerate}	

\end{fact}

\begin{fact} {\rm \cite[Theorems~5.8 and 5.13]{D2012}} \label{MetrProSubs8}
	Let $M$ be a closed linear subspace of $\mathcal{H}$. Then the following statements hold:
	\begin{enumerate}
		\item \label{MetrProSubs8:i}$M^{\perp}$ is a closed linear subspace.
		\item \label{MetrProSubs8:ii} $\Id =\Pro_{M}+\Pro_{M^{\perp}}$. 
		\item  \label{fact:ProjectorInnerPRod:BoundedLine} $\Pro_{M}$ is a bounded linear operator and $\norm{\Pro_{M}}=1$ $($unless $M=\{0\}$, in which case $\norm{\Pro_{M}}=0$$)$.
		\item   \label{fact:ProjectorInnerPRod:Selfadjoint} $\Pro_{M}$ is self-adjoint: $\innp{\Pro_{M}x, y }=\innp{x,\Pro_{M}y}$ for all $x,y$ in $\mathcal{H}$.
	\end{enumerate}
\end{fact}

\begin{fact} {\rm \cite[Theorem~6.24]{D2012} } \label{fact:StrongSeparationSubspace}
	Let $M$ be a closed linear subspace of $\mathcal{H}$ and $x \in \mathcal{H} \smallsetminus M$. Then there exists a point $z \in M^{\perp}$ with $\norm{z}=1$ and $\innp{z,x} >0$.
\end{fact}

\begin{fact} {\rm \cite[Lemma~9.2]{D2012}} \label{fac:proj:commu}
	Let $M$ and $N$ be closed linear subspaces of $\mathcal{H}$. Assume  $M \subseteq N$ or $N \subseteq M$. Then $\Pro_{M}\Pro_{N}=\Pro_{N}\Pro_{M} =\Pro_{M \cap N} $.
\end{fact}

\begin{definition} {\rm \cite[Definition~9.4]{D2012}} \label{defn:FredrichAngleClassical}
	The  \emph{Friedrichs angle} between two linear subspaces $U$ and $V$ is the angle $\alpha(U,V)$ between $0$ and $\frac{\pi}{2}$ whose cosine, $c(U,V) :=\cos \alpha(U,V)$, is defined by the expression
	\begin{align*}
	c(U,V) :=  \sup \{ |\innp{u,v}| ~|~ u \in U \cap (U \cap V)^{\perp}, v \in V \cap (U \cap V)^{\perp}, \norm{u} \leq 1, \norm{v} \leq 1 \}.
	\end{align*}
\end{definition}

\begin{fact}{\rm \cite[Lemma~9.5]{D2012}} \label{fac:PVPUPUcapVperp}
	Let $U$ and $V$ be closed linear subspaces of $\mathcal{H}$. Then $c(U,V)=\norm{\Pro_{V}\Pro_{U}-\Pro_{U \cap V}} = \norm{\Pro_{V}\Pro_{U}\Pro_{(U\cap V)^{\perp}}}$.
\end{fact}

\begin{fact} {\rm \cite[Theorem~9.35]{D2012}} \label{fac:cFLess1}
	Let $U$ and $V$ be closed linear subspaces of $\mathcal{H}$. Then $c(U,V) <1$ if and only if $U +V$ is closed.
\end{fact}

\begin{definition} {\rm \cite[Definition~3.7.5]{BBL1997}} \label{defn:FredrichAngle:N}
	Let $L_{1},\ldots,L_{m}$ be closed linear subspaces of $\mathcal{H}$. Define the \emph{angle $\beta := \beta(L_{1},\ldots, L_{m}) \in [0,\frac{\pi}{2}]$ of the $m$-tuple $(L_{1},\ldots, L_{m})$} by
	\begin{align*}
	\cos \beta := \norm{\Pro_{L_{m}} \cdots \Pro_{L_{1}}\Pro_{(\cap^{m}_{i=1}L_{i})^{\perp}}}.
	\end{align*}
\end{definition}

\begin{fact} {\rm \cite[Proposition~3.7.7]{BBL1997}} \label{fac:AnglN-tupleLess1}
	Let $L_{1},\ldots,L_{m}$ be closed linear subspaces of $\mathcal{H}$. The angle of the $m$-tuple $(L_{1}, \ldots, L_{m})$ is positive if and only if the sum $L^{\perp}_{1} + \cdots + L^{\perp}_{m}$ is closed. 
\end{fact}

\begin{corollary} \label{cor:AnglN-tupleLess1:P}
	Let $L_{1},\ldots,L_{m}$ be closed linear subspaces of $\mathcal{H}$. Then $L^{\perp}_{1} + \cdots + L^{\perp}_{m}$ is closed  if and only if $\norm{\Pro_{L_{m}} \cdots \Pro_{L_{1}}\Pro_{(\cap^{m}_{i=1}L_{i})^{\perp}}} <1$.
\end{corollary}

\begin{proof}
	Combine \cref{defn:FredrichAngle:N} and \cref{fac:AnglN-tupleLess1}.
\end{proof}


\subsection*{Averaged nonexpansive operators}

\begin{definition} \label{defn:Nonexpansive} {\rm \cite[Definition~4.1]{BC2017}}
	Let $D$ be a nonempty subset of $\mathcal{H}$ and let $T:D \rightarrow \mathcal{H}$. Then $T$ is	\begin{enumerate}
		\item \label{FirmNonex} \emph{firmly nonexpansive} if
		\begin{align} \label{EQ:FirmNonex} 
		(\forall x  \in D) (\forall y \in D) \quad \norm{Tx -Ty}^{2} + \norm{(\Id -T)x -(\Id -T)y}^{2} \leq \norm{x -y}^{2};
		\end{align}
		\item \label{Nonex} \emph{nonexpansive} if it is Lipschitz continuous with constant 1, i.e.,
		\begin{align} \label{EQ:Nonex} 
		(\forall x  \in D) 	(\forall y \in D) \quad \norm{Tx -Ty} \leq \norm{x-y};
		\end{align}
		\item \label{FirmlyQuasiNonex} \emph{firmly quasinonexpansive} if
		\begin{align} \label{EQ:FirmlyQuasiNonex} 
		(\forall x \in D)  (\forall y \in \Fix T) \quad \norm{Tx-y}^{2} + \norm{Tx-x}^{2} \leq \norm{x-y}^{2};
		\end{align}
		\item \label{QuasiNonex} \emph{quasinonexpansive} if
		\begin{align} \label{EQ:QuasiNonex} 
		(\forall x \in D)  (\forall y \in \Fix T) \quad \norm{Tx-y} \leq \norm{x-y};
		\end{align}
		\item \label{StrickQuasiNonex} and \emph{strictly quasinonexpansive} if
		\begin{align} \label{EQ:StrickQuasiNonex}
		(\forall x \in D \smallsetminus \Fix T) (\forall y \in \Fix T) \quad \norm{Tx-y} < \norm{x-y}.
		\end{align}
	\end{enumerate}
\end{definition}

\begin{remark}{\rm \cite[page~70]{BC2017}} \label{rem:NonexpImplication}
	Concerning \cref{defn:Nonexpansive}, by definitions we have the implications:
	\begin{align*}
	\text{\cref{EQ:FirmNonex}} \Rightarrow  \text{\cref{EQ:Nonex}} \Rightarrow \text{\cref{EQ:QuasiNonex} } \quad \text{and} \quad
	\text{\cref{EQ:FirmNonex}} \Rightarrow \text{\cref{EQ:FirmlyQuasiNonex}} \Rightarrow  \text{\cref{EQ:StrickQuasiNonex}} \Rightarrow \text{\cref{EQ:QuasiNonex} } .
	\end{align*}
\end{remark}

\begin{definition} {\rm \cite[Definition~4.33]{BC2017}} \label{defn:AlphaAverage}
	Let $D$ be a nonempty subset of $\mathcal{H}$, let $T: D \rightarrow \mathcal{H}$ be nonexpansive, and let $\alpha \in  \left]0,1\right[\,$. Then $T$ is \emph{averaged with constant $\alpha$}, or \emph{$\alpha$-averaged} for short, if there exists a nonexpansive operator $F: D \rightarrow \mathcal{H}$ such that $T=(1- \alpha) \Id +\alpha F$.
\end{definition}

\begin{fact}  {\rm \cite[Remark~4.34(i)$\&$(iii)]{BC2017}} \label{fact:AverFirmNone}
	Let $D$ be a nonempty subset of $\mathcal{H}$, let $T: D \rightarrow \mathcal{H}$.
	\begin{enumerate}
		\item \label{fact:AverFirmNone:nonexp} If $T$ is averaged, then it is nonexpansive.
		\item \label{fact:AverFirmNone:firm} $T$ is firmly nonexpansive if and only if it is $\frac{1}{2}$-averaged.
	\end{enumerate}
\end{fact}

\begin{fact} {\rm \cite[Proposition~4.35]{BC2017}}  \label{fact:averaged:character}
	Let $D$ be a nonempty subset of $\mathcal{H}$, let $T: D \rightarrow \mathcal{H}$ be nonexpansive, and let $\alpha \in \left]0,1\right[\,$. Then the following are equivalent:
	\begin{enumerate}
		\item \label{fact:averaged:character:i} $T$ is $\alpha$-averaged.
		\item \label{fact:averaged:character:ii} $(\forall x \in D)$ $(\forall y \in D)$ $ \norm{Tx -Ty}^{2} +\frac{1-\alpha}{\alpha} \norm{(\Id -T)x -(\Id -T)y}^{2} \leq \norm{x -y}^{2}$.
	\end{enumerate}	
\end{fact}	

\begin{fact} {\rm  \cite[Proposition~4.42]{BC2017} } \label{fact:T:Avera:Sum:FixT}
	Let $D$ be a nonempty subset of $\mathcal{H}$, let $(T_{i})_{i\in \I}$ be a finite family of nonexpansive operators from $D$ to $\mathcal{H}$, let $(\omega_{i})_{i \in \I}$ be real numbers in $\left]0,1\right]$ such that $\sum_{i \in \I} \omega_{i} =1$, and let $(\alpha_{i})_{i \in \I}$ be real numbers in $\left]0,1\right[$ such that, for every $i \in \I$, $T_{i}$ is $\alpha_{i}$-averaged, and set $\alpha := \sum_{i \in \I} \omega_{i} \alpha_{i}$. Then $\sum_{i \in \I} \omega_{i} T_{i}$ is $\alpha$-averaged.
\end{fact}

\begin{fact} {\rm \cite[Proposition~4.47]{BC2017}} \label{fact:FixSumInters}
	Let $D$ be a nonempty subset of $\mathcal{H}$, let $(T_{i})_{i \in \I}$ be a finite family of quasinonexpansive operators from $D$ to $\mathcal{H}$ such that $\cap_{i \in \I} \Fix T_{i} \neq \varnothing$, and let $(\omega_{i})_{i \in \I}$ be strictly positive real numbers such that $\sum_{i \in \I} \omega_{i}=1$. Then $\Fix \sum_{i \in \I} \omega_{i} T_{i}= \cap_{i \in \I} \Fix T_{i}$.
\end{fact}

\begin{fact} {\rm \cite[Proposition~4.49]{BC2017}} \label{fact:FixProdInters}
	Let $D$ be a nonempty subset of $\mathcal{H}$, and let $T_{1}$ and $T_{2}$ be quasinonexpansive operators from $D$ to $D$. Suppose that $T_{1}$ or $T_{2}$ is strictly quasinonexpansive, and that $\Fix T_{1} \cap \Fix T_{2} \neq \varnothing$. Then the following hold:
	\begin{enumerate}
		\item $\Fix T_{1}T_{2} = \Fix T_{1} \cap \Fix T_{2}$. 
		\item Suppose that $T_{1}$ and $T_{2}$ are strictly quasinonexpansive. Then $T_{1}T_{2}$ is strictly quasinonexpansive. 
	\end{enumerate} 
\end{fact}

\begin{lemma} \label{lem:T:alpha}
	Let $T:\mathcal{H} \to \mathcal{H}$ be $\alpha$-averaged with $\alpha \in \left]0,1\right[\,$. Assume that $0 \in \Fix T$. Then 
	\begin{align}  \label{eq:lem:T:alpha}
	(\forall x \in \mathcal{H} \smallsetminus \Fix T) \quad \norm{Tx} < \norm{x}.
	\end{align}	
\end{lemma}	
\begin{proof}
	Since $T$ is $\alpha$-averaged, by \cref{fact:averaged:character},
	\begin{align} \label{eq:lem:T:alpha:Defi}
	(\forall x \in \mathcal{H})  (\forall y \in \mathcal{H}) \quad \norm{Tx -Ty}^{2} + \frac{1-\alpha}{\alpha} \norm{(\Id - T)x-(\Id -T)y}^{2} \leq \norm{x-y}^{2}, 
	\end{align}
	Applying \cref{eq:lem:T:alpha:Defi} with $x \notin \Fix T$ and $y=0$, we obtain \cref{eq:lem:T:alpha}.
\end{proof}	

The following result is motivated by \cite[Lemma~2.1(iv)]{BCS2018}. Moreover, \cref{prop:TFNorm}\cref{prop:TFNorm:TTperp} was shown in \cite[Proposition~2.10]{BOyW2019Isometry}
\begin{proposition} \label{prop:TFNorm}
Suppose that $\mathcal{H}=\mathbb{R}^{n}$. Let $T: \mathcal{H} \to \mathcal{H}$ be linear and  $\alpha$-averaged with $\alpha \in \left]0,1\right[\,$. Then the following assertions hold:
	\begin{enumerate}
		\item   \label{prop:TFNorm:TF} Let $F:  \mathcal{H} \to \mathcal{H}$ be nonexpansive and linear. If $\Fix (T) \cap \Range(F) = \{0\}$, then $\norm{TF} <1$.
		\item  \label{prop:TFNorm:TTperp} $\norm{T\Pro_{(\Fix T)^{\perp}}}<1$.
	\end{enumerate}
\end{proposition}	

\begin{proof}
	\cref{prop:TFNorm:TF}: \cref{lem:T:alpha} implies 
	\begin{align} \label{eq:T:ALPHA:0}
	(\forall x \in \mathbb{R}^{n} \smallsetminus \Fix T) \quad \norm{Tx} < \norm{x}.
	\end{align}
	
	Both $F$ and $T$ are nonexpansive and linear, so $\norm{TF}\leq \norm{T}\norm{F} \leq 1$. Assume to the contrary $\norm{TF} =1$, that is, $1= \norm{TF}  =\max_{\norm{x}=1} \norm{TFx}$. Then there exists $\bar{x} \in \mathcal{H}$ with $\norm{\bar{x}}=1$ and $1= \norm{TF}  =\norm{TF\bar{x}}$. Denote $\hat{x} :=F\bar{x}$. Then $\hat{x} \neq 0$ and $\hat{x} \in \Range{F}$.  By assumption, $\Fix (T) \cap \Range(F) = \{0\}$, so $\hat{x} \notin \Fix T$.
	Substitute $x=\hat{x}$ in \cref{eq:T:ALPHA:0} to obtain that
	\begin{align*}
	1 =\norm{TF\bar{x}} =\norm{T\hat{x}} < \norm{\hat{x}} =\norm{F\bar{x}} \leq \norm{\bar{x}}=1,
	\end{align*}
	which is absurd.
	
	\cref{prop:TFNorm:TTperp}:  By  \cref{MetrProSubs8}\cref{fact:ProjectorInnerPRod:BoundedLine}, $\Pro_{(\Fix T)^{\perp}}$ is nonexpansive and  linear. 
	Moreover, $\Fix T \cap \Range (\Pro_{(\Fix T)^{\perp}} ) =\Fix T \cap(\Fix T)^{\perp} =\{0\} $. Hence, the desired result  is clear by substituting $F=\Pro_{(\Fix T)^{\perp}}$ in \cref{prop:TFNorm:TF}.
\end{proof}	

\begin{fact} {\rm\cite[Page 111--113]{Kreyszig1989} }\label{fact:LinearOperatorMatrix}
 	Let $(X, \norm{\cdot}_{2})$ and $(Y, \norm{\cdot}_{2})$ be finite dimensional real vector
 spaces. Let $E$ and $B$ be bases of $X$ and $Y$ respectively, with the elements of $E$ and $B$ arranged in a definite order (which is arbitrary but fixed).  Let $T : X \to Y$ be a linear operator. Then there exists a matrix $T_{EB}$ uniquely  determined by the linear operator $T$.  We say that the matrix $T_{EB}$  represents the operator $T$ with respect to those bases.  Moreover, $\norm{T} = \norm{T_{EB}}_{2}$. 
\end{fact}	

\begin{fact} {\rm  \cite[Page~281]{MC2000}} \label{fact:2normALambda}
	Let $A \in \mathbb{R}^{n \times m}$.  The matrix $2$-norm induced by the Euclidean vector norm is 
	\begin{align*}
	\norm{A}_{2} = \max_{\norm{x}_{2} \leq 1} \norm{Ax}_{2} = \sqrt{\lambda_{\max}},
	\end{align*}
	where $\lambda_{\max}$ is the largest eigenvalue of $A^{\intercal}A$.
\end{fact}
\begin{proposition} \label{prop:operatornorm:largestgeigenvector}
	Suppose that $\mathcal{H}= \mathbb{R}^{n}$ with the Euclidean norm $\norm{\cdot}_{2}$. Let $T: \mathcal{H} \to \mathcal{H}$ be linear and  $\alpha$-averaged with $\alpha \in \left]0,1\right[\,$.  Assume  that $A$ is a matrix representing of the linear operator $T\Pro_{(\Fix T)^{\perp}}$. Denote the largest eigenvalue of the matrix $A^{\intercal}A$ as $\lambda_{\max}$. Then 
\begin{align*}
\lambda_{\max} = \norm{A}^{2}_{2}  =\norm{T\Pro_{(\Fix T)^{\perp}}}^{2} <1. 
\end{align*}
\end{proposition}	

\begin{proof}
	By \cref{fact:LinearOperatorMatrix}, the matrix $A$ above is well-defined.  Combining  \cref{prop:TFNorm}\cref{prop:TFNorm:TTperp}, \Cref{fact:LinearOperatorMatrix,fact:2normALambda}, we obtain the desired results.
\end{proof}

 \begin{definition} \cite[Definition~3.10-1]{Kreyszig1989}
 Let $T \in \mathcal{B} (\mathcal{H})$ with the \emph{adjoint} $T^{*}$. $T$ is said to be
 	\begin{enumerate}
 		\item  \emph{self-adjoint} if $T^{*} =T$,
 		\item  \emph{unitary} if $T$ is bijective and $T^{*} =T^{-1}$,
 		\item \emph{normal} if $TT^{*} = T^{*}T$.
 	\end{enumerate}
 \end{definition}

\begin{fact} {\rm \cite[Fact~2.25]{BC2017}} \label{fact:RangeKerMore}
	Let $T \in \mathcal{B} (\mathcal{H})$. Then the following statements hold:
	\begin{enumerate}
		\item  \label{fact:RangeKerMore:T} $T^{**} =T$. 
		\item  \label{fact:RangeKerMore:TT2}$\norm{T}  =\norm{T^{*}} = \sqrt{\norm{T^{*}T}}$.
	\item  \label{fact:RangeKerMore:PerpKerRange} $(\ker T)^{\perp} = \overline{\Range} T^{*} $.
		\item  \label{fact:RangeKerMore:KerRange}$(\Range{T} )^{\perp} = \ker T^{*}$.
	\end{enumerate}
\end{fact}

\begin{fact} \cite[Lemma~2.1]{BDHP2003} \label{fact:FixTFixT*}
	Let $T$ be a nonexpansive linear operator on $\mathcal{H}$. Then
	\begin{align*}
	\Fix T =\Fix T^{*}.
	\end{align*}
\end{fact}	

\begin{lemma} \label{lemma:FixRangeNonexpansive}
	Let $T: \mathcal{H} \to \mathcal{H}$ be linear, and nonexpansive. Then 
	\begin{align*}
	\Fix T =( \Range ( \Id - T) )^{\perp}, \quad \text{and} \quad \overline{\Range} \,(\Id -T ) = (\Fix T)^{\perp}.
	\end{align*}
\end{lemma}	

\begin{proof}
	Let $x \in \mathcal{H}$. Clearly, for every operator $F : \mathcal{H} \to \mathcal{H}$, $x \in \Fix F \Leftrightarrow x = Fx \Leftrightarrow (\Id -F) x =0 \Leftrightarrow x \in \ker (\Id -F )$, which implies that 
	\begin{align}  \label{eq:lemma:FixRangeNonexpansive:F}
	\Fix F = \ker (\Id -F ).
	\end{align}
	Because $T$ is nonexpansive and linear and $\Id -T$ is bounded and linear, by \cref{fact:FixTFixT*} and \cref{fact:RangeKerMore}\cref{fact:RangeKerMore:KerRange}, we obtain that
	\begin{align*}
	\Fix T =\Fix T^{*} \stackrel{\cref{eq:lemma:FixRangeNonexpansive:F}}{=} \ker ( \Id - T^{*}  ) =\ker \left(  (\Id-T )^{*} \right) = ( \Range (\Id -T ))^{\perp}.
	\end{align*}
	Similarly, by \cref{fact:FixTFixT*} and \cref{fact:RangeKerMore}\cref{fact:RangeKerMore:PerpKerRange}$\&$\cref{fact:RangeKerMore:T}, we have that
	\begin{align*}
	(\Fix T)^{\perp} =  (\Fix T^{*}  )^{\perp} \stackrel{\cref{eq:lemma:FixRangeNonexpansive:F}}{=}   \left(  \ker ( \Id - T^{*}  ) \right)^{\perp} =  \left( \ker \left(  (\Id -T )^{*} \right) \right)^{\perp} =	\overline{\Range} \, (\Id -T ). 
	\end{align*}
	Therefore, the proof is complete.
\end{proof}

\begin{fact} {\rm \cite[Lemma~2.4]{BDHP2003}} \label{fac:PUmU1FixT}
	Let $U_{1}, \ldots, U_{m}$ be closed linear subspaces of $\mathcal{H}$, and let $T := \Pro_{U_{m}}\Pro_{U_{m-1}} \cdots \Pro_{U_{1}}$. Then $T$ is nonexpansive and
	\begin{align*}
	\Fix T = \Fix T^{*} = \Fix (TT^{*}) =\Fix (T^{*}T) =\cap^{m}_{i=1}U_{i}.
	\end{align*}
\end{fact}

\begin{fact} {\rm \cite[Lemmas~3.14 and 3.15]{BDHP2003}} \label{fac:LineNonexOperNorm}
	Let $T: \mathcal{H} \to \mathcal{H}$ be linear and nonexpansive. Then the following statements hold:
	\begin{enumerate}
		\item  \label{fac:LineNonexOperNorm:i} $(\forall k \in \mathbb{N})$ $\norm{T^{k} - \Pro_{\Fix T}} =\norm{(T\Pro_{(\Fix T)^{\perp}})^{k}} $. In particular,
		\begin{align} \label{eq:fac:LineNonexOperNorm:LineConv}
		(\forall k \in \mathbb{N})  (\forall x \in \mathcal{H}) \quad \norm{T^{k}x - \Pro_{\Fix T}x} \leq \norm{(T\Pro_{(\Fix T)^{\perp}})^{k}}  \norm{x - \Pro_{\Fix T}x}.
		\end{align}
		and $\norm{(T\Pro_{(\Fix T)^{\perp}})^{k}} $ is the smallest constant independent of $x$ for which \cref{eq:fac:LineNonexOperNorm:LineConv} is valid.
		\item \label{fac:LineNonexOperNorm:iv} $\norm{T^{*}T\Pro_{(\Fix T^{*} T)^{\perp}}} \leq \norm{T\Pro_{(\Fix T)^{\perp}}}^{2}$ and $\norm{T^{*}T\Pro_{(\Fix T)^{\perp}}} = \norm{T\Pro_{(\Fix T)^{\perp}}}^{2}$ if $\Fix (T^{*}T)= \Fix T$.
		\item \label{eq:fac:LineNonexOperNorm:normal} If T is normal, then $(\forall k \in \mathbb{N})$ $\norm{T^{k} - \Pro_{\Fix T}} = \norm{(T \Pro_{(\Fix T)^{\perp}} )^{k}}= \norm{ T \Pro_{(\Fix T)^{\perp}}}^{k}$.
		\item \label{cor:LinNoneInequaTstarT} Let $U_{1}, \ldots, U_{m}$ be closed linear subspaces of $\mathcal{H}$, and let $T := \Pro_{U_{m}}\Pro_{U_{m-1}} \cdots \Pro_{U_{1}}$. Then
		\begin{align*}
		(\forall x \in \mathcal{H})  (\forall k \in \mathbb{N}) \quad \norm{(T^{*}T)^{k}x - \Pro_{\cap^{m}_{i=1}U_{i}}x} \leq \norm{T\Pro_{(\cap^{m}_{i=1}U_{i})^{\perp}}}^{2k} \norm{x - \Pro_{\cap^{m}_{i=1}U_{i}}x}.
		\end{align*} 
	\end{enumerate}
\end{fact}

%

\begin{proposition} \label{prop:RnTalphaLinearConv}
	Suppose that $\mathcal{H} = \mathbb{R}^{n}$. Let $T: \mathcal{H} \to \mathcal{H}$ be linear and  $\alpha$-averaged with $\alpha \in \left]0,1\right[\,$. Then $\norm{T\Pro_{(\Fix T)^{\perp}}}<1$ and 
	\begin{align*}
	(\forall x \in \mathcal{H})  (\forall k \in \mathbb{N} ) \quad \norm{T^{k}x - \Pro_{\Fix T}x} \leq \norm{T \Pro_{(\Fix T)^{\perp}} }^{k} \norm{x-\Pro_{\Fix T} x }.
	\end{align*}
	Consequently, $(\forall x \in \mathcal{H})$ $(T^{k}x)_{k \in \mathbb{N}}$ converges to $\Pro_{\Fix T}x$ with a linear rate $\norm{T\Pro_{(\Fix T)^{\perp}}}<1$.
\end{proposition}

\begin{proof}
	$T$ is $\alpha$-averaged implies that $T$ is nonexpansive,  so the required result follows from \cref{fac:LineNonexOperNorm}\cref{fac:LineNonexOperNorm:i} and \cref{prop:TFNorm}\cref{prop:TFNorm:TTperp}.
\end{proof}	

\begin{proposition} \label{prop:selfadjoint:linearConverge}
	Let $F: \mathcal{H} \to \mathcal{H}$ be nonexpansive, linear and normal. Let $\alpha \in \left]0,1\right[\,$. Denote $T :=(1-\alpha)  \Id + \alpha F$. Then the following assertions hold:
	\begin{enumerate}
		\item  \label{prop:selfadjoint:linearConverge:Properties}$T$ is $\alpha$-averaged, linear, and normal. Moreover, we have that $\Fix T=\Fix F$, and that
		\begin{align}  \label{eq:cor:selfadjoint:linearConverge}
		(\forall x \in \mathcal{H})  (\forall k \in \mathbb{N} ) \quad \norm{T^{k}x - \Pro_{\Fix F}x} \leq \norm{T \Pro_{(\Fix F)^{\perp}} }^{k} \norm{x - \Pro_{\Fix F} x }.
		\end{align}
		\item \label{prop:selfadjoint:linearConverge:NormEqual} $(\forall k \in \mathbb{N})$ $\norm{(T \Pro_{(\Fix F)^{\perp}} )^{k}} =\norm{T \Pro_{(\Fix F)^{\perp}} }^{k}$.
		\item \label{prop:selfadjoint:linearConverge:Sharp} Assume that $\mathcal{H} = \mathbb{R}^{n}$. Then
		$(T^{k}x)_{k \in \mathbb{N}}$ converges to $\Pro_{\Fix F}x$ with a sharp linear rate $\norm{T\Pro_{(\Fix F)^{\perp}}}<1$.
	\end{enumerate}
\end{proposition}

\begin{proof}
	\cref{prop:selfadjoint:linearConverge:Properties}:	It is clear that $T$ is $\alpha$-averaged, linear, and  $\Fix T=\Fix F$. The inequality \cref{eq:cor:selfadjoint:linearConverge} follows from \cref{fac:LineNonexOperNorm}\cref{fac:LineNonexOperNorm:i}.  
	Because the normal operators form a vector space which contains $F$ and $\Id$, it is clear that $T$ is normal.

	\cref{prop:selfadjoint:linearConverge:NormEqual}: Combine \cref{prop:selfadjoint:linearConverge:Properties} with \cref{fac:LineNonexOperNorm}\cref{eq:fac:LineNonexOperNorm:normal}.
	
	\cref{prop:selfadjoint:linearConverge:Sharp}:  Combine \cref{prop:selfadjoint:linearConverge:Properties} with \cref{prop:TFNorm}\cref{prop:TFNorm:TTperp} to obtain that $\norm{T\Pro_{(\Fix F)^{\perp}}}<1$.
	Apply \cref{fac:LineNonexOperNorm}\cref{fac:LineNonexOperNorm:i} with \cref{prop:selfadjoint:linearConverge:NormEqual} above  to the linear and nonexpansive operator $T=(1-\alpha)  \Id + \alpha F$, we know that $(\forall k \in \mathbb{N})$ $\norm{T \Pro_{(\Fix F)^{\perp}} }^{k} = \norm{(T \Pro_{(\Fix F)^{\perp}} )^{k}} $ is the smallest constant independent of $x$ for which \cref{eq:cor:selfadjoint:linearConverge} is valid. Therefore, $(T^{k}x)_{k \in \mathbb{N}}$ converges to $\Pro_{\Fix F}x$ with a sharp linear rate $\norm{T\Pro_{(\Fix F)^{\perp}}}<1$.
\end{proof}

\begin{definition} {\rm \cite[Definition~3.1]{BDHP2003}} \label{defn:AccelrAT}
	Let $T: \mathcal{H} \to \mathcal{H}$  be  linear  and nonexpansive. The \emph{accelerated mapping $A_{T}$ of $T$} is defined on $\mathcal{H}$ by
	\begin{align*}
(\forall x \in \mathcal{H}) \quad 	A_{T}(x) :=t_{x} Tx +(1-t_{x}) x,
	\end{align*}
	where
	\begin{align*}
	t_{x} :=t_{x,T} :=  \begin{cases}
\scalebox{1.5}{ $ \frac{\innp{x,x-Tx}}{ \norm{x-Tx}^{2}} $}, \quad \mbox{if}~Tx \neq x;\\
 1, \quad \quad \quad \quad \mbox{if}~Tx = x.
	\end{cases} 
	\end{align*}
\end{definition}

\begin{fact} {\rm \cite[Lemmas~3.27 and 3.8(3)]{BDHP2003}} \label{fac:AT:Norm:MPerp}
	Let $T: \mathcal{H} \to \mathcal{H}$ be linear, nonexpansive, and self-adjoint. 
	Set
	\begin{align} \label{eq:fac:AT:c1:c2:c1}
	c_{1} := \inf \{ \innp{Tx, x} ~|~ x \in (\Fix T)^{\perp}, \norm{x} =1\},
	\end{align}
	and 
	\begin{align} \label{eq:fac:AT:c1:c2:c2}
	c_{2} := \sup \{ \innp{Tx, x} ~|~ x \in (\Fix T)^{\perp}, \norm{x} =1\},
	\end{align}
	where both $c_{1}$ and $c_{2}$ are defined to be $0$ if $(\Fix T)^{\perp} =\{0\}$, i.e., if $\Fix T=\mathcal{H}$. 
	Then
	\begin{align*}
	(\forall x \in \mathcal{H})  (\forall k \in \mathbb{N}) \quad \norm{A^{k}_{T}x - \Pro_{\Fix T}x} \leq \Big(\frac{c_{2}-c_{1}}{2 -c_{1} -c_{2}}\Big)^{k} \norm{x - \Pro_{\Fix T}x}.
	\end{align*}
\end{fact}

\begin{lemma} \label{fac:AT:c1c2:cT}
	Let $T: \mathcal{H} \to \mathcal{H}$ be linear, nonexpansive, self-adjoint and monotone. Let $c_{1}$ and $c_{2}$ be defined as in 
	\cref{eq:fac:AT:c1:c2:c1} and \cref{eq:fac:AT:c1:c2:c2}. Set $c(T) := \norm{T\Pro_{(\Fix T)^{\perp}}}$. Then
	\begin{align*}
	\frac{c_{2}-c_{1}}{2 -c_{1} -c_{2}} = \frac{c(T)-c_{1}}{2 -c_{1} -c(T)} \leq \frac{c(T)}{2 - c(T)}.
	\end{align*}
\end{lemma}

\begin{proof}
	This is inside the proof of  \cite[Theorem~3.29]{BDHP2003}.
\end{proof}

\section{Isometries} \label{sec:Isometries}
In this section, we show some important properties of isometries. Some of them will be used in our main linear convergence results later. 
\begin{definition} \label{defn:isometry} {\rm \cite[Definition~1.6-1]{Kreyszig1989}}
	A mapping $T: \mathcal{H} \rightarrow \mathcal{H}$ is said to be \emph{isometric} or an \emph{isometry} if
	\begin{align} \label{eq:T:normpreserving}
	(\forall x \in \mathcal{H}) (\forall y \in \mathcal{H}) \quad \norm{Tx -Ty} =\norm{x-y}.
	\end{align}
\end{definition}

Note that in some references, the definition of isometry is the linear operator satisfying \cref{eq:T:normpreserving}. In this paper, the definition of isometry follows  from \cite[Definition~1.6-1]{Kreyszig1989} where the linearity is not required. 

We show some common isometries in the following fact.

\begin{fact} \label{fact:examples:normpreserving} {\rm \cite[Lemmas~2.23 and 2.24]{BOyW2019Isometry} }
	\begin{enumerate}
		\item \label{fact:examples:normpreserving:R} Let $C$ be a closed affine subspace of $\mathcal{H}$. Then the reflector $\R_{C}=2\Pro_{C}-\Id$ is isometric with $\Fix \R_{C} =C$.
		\item \label{fact:examples:normpreserving:Trans} Let $a \in \mathcal{H}$. The translation operator $(\forall x \in \mathcal{H})$ $T_{a}x=x+a$ is isometric.
		\item \label{fact:examples:normpreserving:TTStar} Let $T \in \mathcal{B}(\mathcal{H},\mathcal{H})$ and let $T^{*}$ be the adjoint of $T$. Then $T$ is isometric  if and only if $T^{*}T =\Id$.
		\item \label{fact:examples:normpreserving:Id}The identity operator is isometric.
		\item \label{fact:examples:normpreserving:composition} The composition of finitely many isometries is an isometry.
	\end{enumerate}
\end{fact}
Clearly, the reflector associated with an affine subspace is affine but not necessarily linear. The translation operator $T_{a}$ defined in \cref{fact:examples:normpreserving}\cref{fact:examples:normpreserving:Trans} is not linear and $\Fix T_{a} = \varnothing$ whenever $a \neq 0$.

\begin{fact} {\rm \cite[Lemma~1.7]{Zarantonello1971}}{\rm \cite[Proposition~4.8(iii)]{BC2017}}\label{fact:TtobeAffine}
	Let $D$ be a nonempty convex subset of $\mathcal{H}$, let $T: D \to \mathcal{H}$, let $(x_{i})_{i \in \I}$ be a finite family in $D$,  let $(\alpha_{i})_{i \in \I}$ be a finite family in $\mathbb{R}$ s.t. $\sum_{i \in \I} \alpha_{i} =1$, and set $y :=\sum_{i \in \I} \alpha_{i}x_{i}$. Then $\norm{Ty - \sum_{i \in \I} \alpha_{i} Tx_{i} }^{2} +\sum_{i \in \I} \alpha_{i} \left( \norm{y -x_{i}}^{2} - \norm{Ty - Tx_{i}}^{2} \right) = \frac{1}{2} \sum_{i \in \I} \sum_{j \in \I} \alpha_{i} \alpha_{j}  \left( \norm{x_{i}-x_{j}}^{2} - \norm{Tx_{i} - Tx_{j}}^{2} \right) $.
\end{fact}
\begin{proposition} \label{prop:IsometryAffine}
	Let $T: \mathcal{H} \rightarrow \mathcal{H}$ be isometric. Then $T$ is affine.  
\end{proposition}	

\begin{proof}
	The desired result is directly from \cref{defn:isometry} and \cref{fact:TtobeAffine}.
\end{proof}	

\begin{corollary}
	Let $T: \mathcal{H} \rightarrow \mathcal{H}$ be isometric. If $\Fix T$ is nonempty, then $\Fix T$ is an affine closed subspace. 
	
	Consequently, the intersection of  the fixed point sets of finitely many isometries is either empty or an affine closed subspace. 
\end{corollary}
\begin{proof}
	The desired result is easily from the related definitions and \cref{prop:IsometryAffine}.
\end{proof}

\begin{fact} {\rm \cite[Page~321]{MC2000}} \label{fact:isometry:orthogn}
	The linear isometries on $\mathbb{R}^{n}$ are precisely the orthogonal matrices.
\end{fact}

\begin{lemma} \label{lema:IsometryT}
	Let $T: \mathcal{H} \to \mathcal{H}$ be isometric and let $W$ be a nonempty closed   convex set such that $W \subseteq \Fix T$. Then 
	\begin{align*}
	T \Pro_{W} =\Pro_{W} = \Pro_{W} T.
	\end{align*} 
\end{lemma}	

\begin{proof}
	Let $x \in \mathcal{H}$. Because $ \Pro_{W} x \in W \subseteq \Fix T$ and $ \Pro_{W} Tx \in W \subseteq \Fix T$, $T \Pro_{W}x = \Pro_{W}x$ and $ T\Pro_{W}Tx =\Pro_{W}Tx$. Hence, $T \Pro_{W} =\Pro_{W}$. Moreover, by definitions of projection and isometry, we have that $  \norm{ Tx - \Pro_{W}Tx } \leq \norm{ Tx - \Pro_{W} x} = \norm{ Tx -T \Pro_{W} x} = \norm{x -\Pro_{W}x} \leq \norm{x - \Pro_{W}Tx} = \norm{Tx - T\Pro_{W}Tx} = \norm{Tx - \Pro_{W}Tx}$, which implies that $ \norm{ Tx - \Pro_{W}Tx } = \norm{ Tx - \Pro_{W} x}  $. By the uniqueness of projection on the nonempty closed   convex set $W$, we obtain that $\Pro_{W}x =\Pro_{W}Tx $. Hence, $\Pro_{W} =\Pro_{W}T$.
\end{proof}

\begin{lemma} \label{lemma:TAffineFLinear}
	Let $z \in \mathcal{H}$ and let
	$T: \mathcal{H} \to \mathcal{H}$ and $F: \mathcal{H} \to \mathcal{H} $ such that $(\forall x \in \mathcal{H})$ $Fx = T(x +z)-z$. Then the following statements hold:
	\begin{enumerate}
		\item \label{lemma:TAffineFLinear:FT}  If $F$ is affine, then $T$ is affine.
		\item \label{lemma:TAffineFLinear:TF} 
		Suppose that $z \in \Fix T$. If $T$ is affine, then $F$ is linear. 
		\item \label{lemma:TAffineFLinear:Fix}$\Fix F =\Fix T -z$.
	    \item \label{lemma:TAffineFLinear:Isometric} $T$ is isometric if and only if $F$ is isometric.
	\end{enumerate}	
\end{lemma}	

\begin{proof} Let $x, y$ be in $\mathcal{H}$.
	
	\cref{lemma:TAffineFLinear:FT}:  Let $\lambda$ be in  $\mathbb{R}$. Because  $F$ is affine, we have 
	\begin{align*}
	T(\lambda x + (1-\lambda)y) &= z + F(\lambda x + (1-\lambda)y -z ) \\
	& = z + F(\lambda (x -z) + (1-\lambda)(y -z) ) \\
	& = \lambda ( z + F(x-z)) + (1-\lambda )  ( z + F(y-z))\\
	& = \lambda T x + (1-\lambda) Ty.
	\end{align*}
	
	\cref{lemma:TAffineFLinear:TF}: Let $\alpha, \beta$ be in $\mathbb{R}$. Because $T$ is affine, 
	\begin{align*}
	F(\alpha x + \beta y) &= T(\alpha x + \beta y +z ) -z\\
	&= T\big( \frac{1}{2} ( 2\alpha x +z)  + \frac{1}{2} ( 2 \beta y +z) \big) -z\\
	&= \frac{1}{2}  T(2 \alpha x +z)  + \frac{1}{2}  T(2 \beta y +z)  -z \\
	&= \frac{1}{2}  T(2 \alpha x +z)  + \frac{1}{2}  Tz + \frac{1}{2}  T(2 \beta y +z)  + \frac{1}{2}  Tz -2z \quad (\text{by $z \in \Fix T$})\\
	&= T( \alpha x +z) + T( \beta y +z) -2z \\
	&=   T \left( \alpha (x+z) +(1- \alpha)z \right) + T\left( \beta (y+z) +(1-\beta)z \right) -2z \\
	&= \alpha T(x+z) +(1- \alpha) Tz + \beta T (y+z) + (1-\beta) Tz -2z \\
	&=  \alpha  \left( T(x+z) -z\right)  + \beta ( T (y+z)  -z )\\
	&= \alpha F x + \beta Fy.
	\end{align*}
	Hence, $F$ is linear.
	
	\cref{lemma:TAffineFLinear:Fix}: Clearly, $x \in \Fix F \Leftrightarrow x =Fx = T(x+z) -z   \Leftrightarrow x+z = T(x+z)   \Leftrightarrow x+z \in \Fix T$.

	\cref{lemma:TAffineFLinear:Isometric}: This is clear from $\norm{Tx- Ty} =\norm{ z+F(x-z)- \left(z+F(y-z)  \right)} = \norm{F(x-z)-F(y-z)} $.
\end{proof}

\subsection*{Properties of surjective or self-adjoint linear isometries}

\begin{lemma}\label{lemma:TRnUnitaryNormal}
	Let $T: \mathbb{R}^{n} \to \mathbb{R}^{n}$ be a linear isometry. Then $T$ is unitary and normal.
\end{lemma}
\begin{proof}
	By \cref{fact:isometry:orthogn}, without loss of generality, we assume that $T \in \mathbb{R}^{n \times n}$ is an orthogonal matrix. Hence, by \cite[Page~321]{MC2000}, $T$ has orthonormal columns and orthonormal rows, which implies that $ T^{\intercal}T =\Id =TT^{\intercal}$. Therefore, $T$ is unitary and normal.
\end{proof}

The last result states that linear isometries on $\mathbb{R}^{n}$ must be normal; however, this fails in infinite-dimensional Hilbert space. 
\begin{example} \label{exam:CounterExam:isometry}
	Suppose that $\mathcal{H} =\ell^{2} = \{ (x_{i})_{i \in \mathbb{N}} ~:~ \sum_{i \in \mathbb{N}} x^{2}_{i}  <\infty ~\text{and}~ (\forall i \in \mathbb{N}) ~x_{i} \in \mathbb{R} \}$ with the inner product $\innp{x,y} = \sum_{i \in \mathbb{N}} x_{i}y_{i}$ for every $x= (x_{i})_{i \in \mathbb{N}}$ and $y=(y_{i})_{i \in \mathbb{N}}$ in $\ell^{2}$. Define the right shift operator $T_{R}$ and left shift operator $T_{L}$ by 
	\begin{align*}
	\big(\forall x= (x_{i})_{i \in \mathbb{N}} \in \ell^{2}\big) \quad T_{R}x := (0,x_{0},x_{1},x_{2}, \cdots),
	\end{align*}
	and
	\begin{align*}
	\big(\forall x= (x_{i})_{i \in \mathbb{N}} \in \ell^{2}\big) \quad T_{L}x := (x_{1},x_{2}, x_{3},x_{4}, \cdots).
	\end{align*}
	Then the following assertions hold:
	\begin{enumerate}
		\item $T_{L}$ and $T_{R}$ are linear.
		\item $T_{R}^{*}= T_{L} \neq T_{R}$.
		\item $T_{R}^{*}T_{R}=T_{L}T_{R} = \Id$, but $T_{L}^{*}T_{L}=T_{R}T_{L} \neq \Id$. Hence, $T_{R}^{*}T_{R} = \Id \neq T_{R}T_{R}^{*}$.
		\item $T_{R}$ is isometric, but $T_{L}=T_{R}^{*}$ is not isometric.
		\item $T_{R}$ is not normal.
		\item $T_{R}$ is not surjective. Hence, $T_{R}$ is not unitary.
	\end{enumerate}
\end{example}
\begin{remark}
Recall that in the Hilbert sequence space	$\ell^{2}$, we draw from \cref{exam:CounterExam:isometry} the following conclusions:
\begin{enumerate}
	\item A linear isometry need not be self-adjoint.
	\item A linear isometry need not be surjective; hence, a linear isometry need not be unitary. 
	\item Even if $T$ is linear and isometric, $T^{*}$ may fail to be isometric. 
	\item A linear isometry need not be normal.
\end{enumerate}

\begin{corollary} \label{cor:RnFMF2F1}
	Suppose that $\mathcal{H} = \mathbb{R}^{n}$.  Let $F_{1}, F_{2},  \ldots, F_{m}$ be linear  isometries on $ \mathbb{R}^{n}$. Set $T:=(1-\alpha)  \Id  +  \alpha F_{m}\cdots F_{2}F_{1}$  with $\alpha \in \left]0,1\right[\,$. Then $\Fix T = \Fix F_{m}\cdots F_{2}F_{1}$. Moreover, for every $x \in  \mathbb{R}^{n}$, $(T^{k}x)_{k \in \mathbb{N}}$ converges to $\Pro_{\Fix F_{m}\cdots F_{2}F_{1}}x$ with a sharp linear rate $\norm{T\Pro_{(\Fix F_{m}\cdots F_{2}F_{1})^{\perp}}}<1$.
\end{corollary}	

\begin{proof}
	Because $F_{m}\cdots F_{2}F_{1}$  is a linear isometry on $ \mathbb{R}^{n}$, the result comes from \cref{lemma:TRnUnitaryNormal} and \cref{prop:selfadjoint:linearConverge}\cref{prop:selfadjoint:linearConverge:Sharp}.
\end{proof}	

\begin{example}
	Suppose that $\mathcal{H} = \mathbb{R}^{n}$. Let $U_{1}, U_{2}$ be linear subspaces. Denote by $T:=  \frac{1}{2}\Id  + \frac{1}{2}  \R_{U_{2}}\R_{U_{1}}$, the Douglas--Rachford operator. Then $(T^{k}x)_{k \in \mathbb{N}}$  converges linearly to $\Pro_{\Fix T}x$ with a sharp linear rate $\norm{T\Pro_{(\Fix T)^{\perp}}}<1$.
\end{example}

More relations among isometric, normal and unitary operators can be found in \cite[Section~3.10]{Kreyszig1989}.
\end{remark}
\begin{theorem} \label{theo:TselfAdjoinIsome:ProFixT}
	Let $T\in \mathcal{B}( \mathcal{H}, \mathcal{H})$. Then the following statements hold:
	\begin{enumerate}
		\item \label{theo:TselfAdjoinIsome:ProFixT:ProT} If $T$ is  isometric and self-adjoint, then $\Pro_{\Fix T}=\frac{1}{2}(\Id +T)$ and $ \Pro_{(\Fix T)^{\perp}}=\frac{1}{2}(\Id - T) $.
		\item  \label{theo:TselfAdjoinIsome:ProFixT:RT} $T$ is isometric and self-adjoint if and only if  $T=\R_{U}$, where $U$ is a closed linear subspace of $\mathcal{H}$.
	\end{enumerate} 
\end{theorem}	

\begin{proof}
	\cref{theo:TselfAdjoinIsome:ProFixT:ProT}: Suppose that $T$ is isometric and self-adjoint. Then $\Fix T$ is a closed linear subspace of $\mathcal{H}$, 
	\begin{align}  \label{eq:lemma:TselfAdjoinIsome:ProFixT:T}
	T=T^{*}, \quad \text{and} \quad  T^{2} = T^{*}T=\Id
	\end{align}
	 by \cref{fact:examples:normpreserving}\cref{fact:examples:normpreserving:TTStar}. 
	 Let $x \in \mathcal{H}$. Then
	\begin{align*}
	T \left(\tfrac{1}{2}(\Id +T) x \right)=\tfrac{1}{2} (Tx +TTx)\stackrel{\cref{eq:lemma:TselfAdjoinIsome:ProFixT:T}}{=}  \tfrac{1}{2}(Tx +x),
	\end{align*}
	and so $\tfrac{1}{2}(\Id +T) x \in \Fix T$. Moreover, 
	\begin{align*}
	(\forall y \in \Fix T) ~ \Innp{x- \tfrac{1}{2}(\Id +T) x,y} = \tfrac{1}{2} \innp{x - Tx, y}=   \tfrac{1}{2} \left( \innp{x , y}  - \innp{Tx,y} \right) \stackrel{\cref{eq:lemma:TselfAdjoinIsome:ProFixT:T}}{=}   \tfrac{1}{2} \left( \innp{x , y}  - \innp{x, Ty} \right) =  0.
	\end{align*}
	Hence, by	\cref{fact:CharaProjectionLineaSpace}, we obtain that  $\Pro_{\Fix T}x=\frac{1}{2}(\Id +T)x$.
	
	Because $\Fix T$ is a closed linear subspace,  by \cref{MetrProSubs8}\cref{MetrProSubs8:ii},
	$ \Pro_{(\Fix T)^{\perp}}=\Id -  \Pro_{\Fix T}= \frac{1}{2}(\Id - T) $.
	
	\cref{theo:TselfAdjoinIsome:ProFixT:RT}:  \enquote{$\Rightarrow$} By \cref{theo:TselfAdjoinIsome:ProFixT:ProT}, $T=2  \Pro_{\Fix T} - \Id =\R_{\Fix T}$ is the reflector associated with the closed linear subspace $\Fix T$. 
	
	\enquote{$\Leftarrow$}  By \cref{fact:ProjectorInnerPRod}\cref{fact:ProjectorInnerPRod:Idempotent} and  \cref{MetrProSubs8}\cref{fact:ProjectorInnerPRod:BoundedLine}$\&$\cref{fact:ProjectorInnerPRod:Selfadjoint}, we know that $\Pro^{2}_{U} =\Pro_{U}$, $ \Pro_{U} \in \mathcal{B}(\mathcal{H})$ and $\Pro_{U}=\Pro^{*}_{U}$. Hence, $\R_{U}=2\Pro_{U} -\Id$ satisfies $\R_{U} \in \mathcal{B}(\mathcal{H}) $,  $\R_{U} = \R^{*}_{U}$ and $ \R^{*}_{U}\R_{U} =\Id$. By \cref{fact:examples:normpreserving}\cref{fact:examples:normpreserving:TTStar}, the proof is complete.
\end{proof}	

\begin{fact} {\rm \cite[Proposition~29.6]{BC2017} } \label{fact:UVperp}
	Let $C$ and  $D$ be nonempty closed convex subsets of $\mathcal{H}$ such that $C \perp D$. Then $C+D$ is closed. 
\end{fact}

The following result is essentially \cite[Proposition~3.6]{BCNPW2014}, but our proof is different. 
\begin{theorem} \label{theo:LinearSelfAdIsometry:Fix}
	Let $T_{1} : \mathcal{H} \to \mathcal{H}$ and $T_{2}:  \mathcal{H} \to \mathcal{H}$ be linear, self-adjoint and isometric. Then 
	\begin{align*}
	\Fix T_{2}T_{1} = \left( \Fix T_{1} \cap \Fix T_{2} \right)  \oplus \big( (\Fix T_{1})^{\perp} \cap ( \Fix T_{2})^{\perp} \big).
	\end{align*} 
\end{theorem}	

\begin{proof}
	Clearly, $\left( \Fix T_{1} \cap \Fix T_{2} \right) \cap \left( (\Fix T_{1})^{\perp} \cap ( \Fix T_{2})^{\perp} \right) =\{0\} $. Hence, it remains to prove $\Fix T_{2}T_{1} = \left( \Fix T_{1} \cap \Fix T_{2} \right)  + \left( (\Fix T_{1})^{\perp} \cap ( \Fix T_{2})^{\perp} \right)$.
	First note that
	\begin{subequations}
		\begin{align}
		(\forall x \in \mathcal{H})	\quad 	x \in \Fix T_{2}T_{1} & \Leftrightarrow x = T_{2}T_{1} x \\
		& \Leftrightarrow T_{2} x =T_{1}x \quad (T^{*}_{2}=T_{2} ~\text{and}~T^{*}_{2}T_{2} = \Id )\\
		&\Leftrightarrow 2 \Pro_{\Fix T_{2}} x-x =2 \Pro_{\Fix T_{1}} x -x \quad (\text{by \cref{theo:TselfAdjoinIsome:ProFixT}\cref{theo:TselfAdjoinIsome:ProFixT:ProT}}) \\
		& \Leftrightarrow  \Pro_{\Fix T_{2}} x =\Pro_{\Fix T_{1}} x  \label{eq:FixT1T2P} \\
		& \Leftrightarrow   \Pro_{(\Fix T_{2})^{\perp}} x =\Pro_{(\Fix T_{1})^{\perp}} x. \quad  (\text{by  \cref{MetrProSubs8}\cref{MetrProSubs8:ii}}) \label{eq:FixT1T2PPerp}
		\end{align}
	\end{subequations}
	Clearly, \cref{eq:FixT1T2P} and \cref{eq:FixT1T2PPerp} respectively imply that 
	\begin{align*} 
	\Fix T_{1} \cap \Fix T_{2} \subseteq  \Fix T_{2}T_{1} \quad \text{and} \quad (\Fix T_{1})^{\perp} \cap  ( \Fix T_{2})^{\perp}  \subseteq  \Fix T_{2}T_{1}. 
	\end{align*}
	Since  $\Fix T_{2}T_{1}$ is a closed linear subspace of $\mathcal{H}$, we have
	\begin{align*}
	\left( \Fix T_{1} \cap \Fix T_{2} \right)  + \big( (\Fix T_{1})^{\perp} \cap ( \Fix T_{2})^{\perp} \big) \subseteq \Fix T_{2}T_{1}.
	\end{align*}
	
	It suffices to show that $\Fix T_{2}T_{1} \subseteq  \left( \Fix T_{1} \cap \Fix T_{2} \right)  + \left( (\Fix T_{1})^{\perp} \cap ( \Fix T_{2})^{\perp} \right)  $.
	Let $x \in \Fix T_{2}T_{1}$ but assume to the contrary that $x \notin  \left( \Fix T_{1} \cap \Fix T_{2} \right)  + \left( (\Fix T_{1})^{\perp} \cap ( \Fix T_{2})^{\perp} \right) $.
	By \cref{fact:UVperp}, we know that $\left( \Fix T_{1} \cap \Fix T_{2} \right)  + \left( (\Fix T_{1})^{\perp} \cap ( \Fix T_{2})^{\perp} \right)  $ is a closed linear subspace of $\mathcal{H}$. Then by \cref{fact:StrongSeparationSubspace}, there exists $z \in \mathcal{H}$ with $\norm{z} =1$ such that 
	\begin{align} \label{eq:xlarger0}
	\innp{z,x} >0,
	\end{align}
	and
	\begin{align} \label{eq:yequal0}
	\left(\forall y \in  \left( \Fix T_{1} \cap \Fix T_{2} \right)  + \big( (\Fix T_{1})^{\perp} \cap ( \Fix T_{2})^{\perp} \big)  \right) \quad \innp{z, y} =0.
	\end{align} 
	By \cref{eq:FixT1T2P} and \cref{eq:FixT1T2PPerp}, we have 
	\begin{align} \label{eq:UANDUPERP}
	\Pro_{\Fix T_{1}}x =\Pro_{\Fix T_{2}}x \in   \Fix T_{1} \cap \Fix T_{2} \quad \text{and} \quad  \Pro_{(\Fix T_{1} )^{\perp}}x =\Pro_{(\Fix T_{2} )^{\perp}}x \in (\Fix T_{1} )^{\perp}\cap (\Fix T_{2} )^{\perp}.
	\end{align}
	Combine \cref{MetrProSubs8}\cref{MetrProSubs8:ii} with  \cref{eq:UANDUPERP}  to obtain that
	\begin{align*}
	x = \Pro_{\Fix T_{1}}x + \Pro_{(\Fix T_{1} )^{\perp}}x \in \left( \Fix T_{1} \cap \Fix T_{2} \right)+ \big( (\Fix T_{1} )^{\perp} \cap (\Fix T_{2} )^{\perp} \big).
	\end{align*}
	Hence, by \cref{eq:yequal0}, we have $\innp{z,x} =0$ which contradicts with \cref{eq:xlarger0}.	 
\end{proof}

\begin{corollary} \label{coro:RNT}
	Let $U_{1}, U_{2}$ be  closed linear subspaces of $\mathcal{H}$. Let $T:=\frac{1}{2} (\Id +  \R_{U_{2}}\R_{U_{1}} )$ be the Douglas--Rachford operator. Then $\Fix T =\Fix  \R_{U_{2}}\R_{U_{1}}= \left( U_{1} \cap U_{2}  \right) \oplus \left( U^{\perp}_{1} \cap U^{\perp}_{2}  \right)$.
\end{corollary}
\begin{proof}
	The result follows from  \cref{theo:LinearSelfAdIsometry:Fix} and \cref{theo:TselfAdjoinIsome:ProFixT}\cref{theo:TselfAdjoinIsome:ProFixT:RT}.
\end{proof}

The following examples show that it is not clear how to generalize \cref{theo:LinearSelfAdIsometry:Fix} from two to finitely many isometries.
\begin{example}
	 Let $U_{1}, U_{2}$ be linear subspaces of $\mathbb{R}^{n}$ with $U^{\perp}_{1} \cap U^{\perp}_{2}  \neq \{0\}$, and let $U_{3} :=\mathbb{R}^{n}$. Then
	\begin{align*}
	\Fix \R_{U_{3}}\R_{U_{2}}\R_{U_{1}} = \Fix \R_{U_{2}}\R_{U_{1}}= \left( U_{1} \cap U_{2}  \right)+ \big( U^{\perp}_{1} \cap U^{\perp}_{2}  \big)  \neq \left( U_{1} \cap U_{2}  \cap U_{3} \right)+ \left( U^{\perp}_{1} \cap U^{\perp}_{2} \cap U^{\perp}_{3}  \right).
	\end{align*}
\end{example}

\begin{example} \label{exam:RU3U2U1}
	Suppose that $\mathcal{H}=\mathbb{R}^{2}$. Let $U_{1} := \mathbb{R}\cdot(1,0)$, $U_{2} := \mathbb{R}\cdot(1,1)$ and $U_{3} :=\mathbb{R} \cdot (0,1)$. Then
	\begin{align*}  
	\Fix \R_{U_{3}}\R_{U_{2}}\R_{U_{1}} = \mathbb{R}\cdot(1,1).
	\end{align*}
	Consequently, 
	\begin{align*}
	\Fix \R_{U_{3}}\R_{U_{2}}\R_{U_{1}} \neq  U_{1} \cap U_{2} \cap U_{3} \quad \text{and} \quad \Fix \R_{U_{3}}\R_{U_{2}}\R_{U_{1}} \neq \left( U_{1} \cap U_{2}  \cap U_{3} \right)+ \left( U^{\perp}_{1} \cap U^{\perp}_{2} \cap U^{\perp}_{3}  \right).
	\end{align*}
\end{example}

\section{Circumcentered isometry methods} \label{sec:CircumcenteredIsometryMethods}

\subsection*{Circumcenter mappings}
Recall that  $\mathcal{P}(\mathcal{H})$ is the set of all nonempty subsets of $\mathcal{H}$ containing \emph{finitely many} elements.
By  \cite[Proposition~3.3]{BOyW2018},  the following definition is well defined.
\begin{definition}[circumcenter operator]  {\rm \cite[Definition~3.4]{BOyW2018}} \label{defn:Circumcenter}
	The \emph{circumcenter operator} is
	\begin{empheq}[box=\mybluebox]{equation*}
	\CCO{} \colon \mathcal{P}(\mathcal{H}) \to \mathcal{H} \cup \{ \varnothing \} \colon K \mapsto \begin{cases} p, \quad ~\text{if}~p \in \aff (K)~\text{and}~\{\norm{p-y} ~|~y \in K \}~\text{is a singleton};\\
	\varnothing, \quad~ \text{otherwise}.
	\end{cases}
	\end{empheq}
	In particular, when $\CCO(K) \in \mathcal{H}$, that is, $\CCO(K) \neq \varnothing$, we say that the circumcenter of $K$ exists and we call $\CCO(K)$ the \emph{circumcenter of $K$}.
\end{definition}

\begin{fact}[scalar multiples]  {\rm \cite[Proposition~6.1]{BOyW2018}} \label{fact:CircumHomoge}
	Let $K \in \mathcal{P}(\mathcal{H})$ and $\lambda \in \mathbb{R} \smallsetminus \{0\}$.
	Then
	$\CCO(\lambda K)=\lambda \CCO(K)$.
\end{fact}

\begin{fact}[translations]  {\rm \cite[Proposition~6.3]{BOyW2018}} \label{fact:CircumSubaddi}
	Let $K \in \mathcal{P}(\mathcal{H})$ and $y \in \mathcal{H}$. Then
	$\CCO(K+y)=\CCO(K)+y$.
\end{fact}

Throughout this subsection, we assume that  
\begin{empheq}[box=\mybluebox]{equation*}
G_{1}, \ldots, G_{m} ~\text{are operators from}~\mathcal{H} ~\text{to}~\mathcal{H}~\text{with}~\cap^{m}_{j=1} \Fix G_{j} \neq \varnothing,
\end{empheq}
and that
\begin{empheq}[box=\mybluebox]{equation*}
\mathcal{S}=\{ G_{1}, \ldots,  G_{m} \} \quad \text{and} \quad (\forall x \in \mathcal{H}) \quad \mathcal{S}(x)=\{ G_{1}x, \ldots,  G_{m}x\} .
\end{empheq}

\begin{definition}[circumcenter mapping] {\rm \cite[Definition~3.1]{BOyW2018Proper}} \label{def:cir:map}
	The \emph{circumcenter mapping} induced by $\mathcal{S}$ is
	\begin{empheq}[box=\mybluebox]{equation*}
	\CC{\mathcal{S}} \colon \mathcal{H} \to \mathcal{H} \cup \{ \varnothing \} \colon x \mapsto \CCO(\mathcal{S}(x)),
	\end{empheq}
	that is, for every $x \in \mathcal{H}$, if the circumcenter of the set $\mathcal{S}(x)$ defined in \cref{defn:Circumcenter} does not exist, then  $\CC{\mathcal{S}}x= \varnothing $. Otherwise, $\CC{\mathcal{S}}x$ is the \emph{unique} point satisfying the two conditions below:
	\begin{enumerate}
		\item $\CC{\mathcal{S}}x \in \aff(\mathcal{S}(x))=\aff\{G_{1}(x), \ldots,    G_{m}(x)\}$, and
		\item $\left\{ \norm{\CC{\mathcal{S}}x - G_{i}(x)}~\big|~ i \in \{1, \ldots, m\} \right\}$ is a singleton, that is,
		\begin{align*}
		\norm{\CC{\mathcal{S}}x -G_{1}(x)}=\cdots =\norm{\CC{\mathcal{S}}x -G_{m}(x)}.
		\end{align*}
	\end{enumerate}
	In particular, if for every $x \in \mathcal{H}$, $\CC{\mathcal{S}}x \in \mathcal{H}$, then we say the circumcenter mapping $\CC{\mathcal{S}}$ induced by $\mathcal{S}$ is \emph{proper}. Otherwise, we call  $\CC{\mathcal{S}}$ \emph{improper}.
\end{definition}

Assume that $\CC{\mathcal{S}}$ is proper. Recall that the \emph{circumcenter method} induced by $\mathcal{S}$ is
\begin{empheq}[box=\mybluebox]{equation} \label{eq:CIMSequence}
x_{k} :=\CC{\mathcal{S}}(x_{k-1})=\CC{\mathcal{S}}^{k}x_{0}, ~\mbox{where}~x_{0} \in \mathcal{H}~\text{and}~k=1,2,\ldots.
\end{empheq}

\begin{fact}{\rm \cite[Proposition~2.33]{BOyW2019Isometry}} \label{fact:CW:Welldefined:Formula}
	Assume $\CC{\mathcal{S}}$ is proper. Then there exist functions $(\forall i \in \{1,\ldots,m-1\})$ $\alpha_{i}:\mathcal{H} \to \mathbb{R}$ such that 
	\begin{align*}
	(\forall x \in \mathcal{H}) \quad	\CC{\mathcal{S}}x= 
	G_{1}x+ \sum^{m-1}_{i=1} \alpha_{i}(x) (G_{i+1}x-G_{1}x).
	\end{align*}
\end{fact}

\begin{fact} {\rm \cite[Proposition~3.10]{BOyW2018Proper}} \label{fact:FixCCS:F1Ft:T}
	The following hold:
	\begin{enumerate}
	\item \label{fact:CW:FixPointSet:equa} 
	If $\Fix \CC{\mathcal{S}} \subseteq \cup^{m}_{i=1} \Fix G_{i}$, 
	then $\Fix \CC{\mathcal{S}} = \cap^{m}_{i=1} \Fix G_{i}$.
	\item \label{fact:Id:FixPointSet:equa} If $\Id  \in \mathcal{S}$, 
	then $\Fix \CC{\mathcal{S}} = \cap^{m}_{i=1} \Fix G_{i}$.
\end{enumerate}
\end{fact}

\begin{lemma}
	Let $z \in \mathcal{H}$ and $\alpha \in \mathbb{R}$. Define $(\forall i \in \{1,\ldots, m\})$ $(\forall x \in \mathcal{H})$ $F_{i}x := \alpha G_{i}x- z$, and set  $\widehat{\mathcal{S}} :=\{F_{1}, F_{2}, \ldots, F_{m} \}$. Then  $(\forall x \in \mathcal{S})$ $\CC{\widehat{\mathcal{S}}}x = \alpha \CC{\mathcal{S}} (x)-z$.
\end{lemma}	
\begin{proof}
If $\alpha =0$, the result is trivial. Assume $\alpha \neq 0$.	By 	 \cref{def:cir:map} and by \Cref{fact:CircumHomoge,fact:CircumSubaddi} ,
	\begin{align*}
	(\forall x \in \mathcal{H}) \quad \CC{\widehat{\mathcal{S}}}x & = \CCO{(\widehat{\mathcal{S}}(x) )} \\
	& = \CCO{\left( \{  F_{1}x, F_{2}x, \ldots, F_{m}x  \}  \right) } \\
	& = \CCO{ \left(  \{ \alpha G_{1}x -z , \alpha G_{2}x -z , \ldots, \alpha G_{m}x -z  \} \right)  } \\
	& = \CCO{ \left(  \alpha \{  G_{1}x, G_{2}x, \ldots, G_{m}x \} -z \right)  } \\
	&= \alpha \CCO{ \left( \{  G_{1}x , G_{2}x, \ldots, G_{m}x\} \right) } -z \\
	&=\alpha \CCO{(\mathcal{S}(x))} -z= \alpha \CC{\mathcal{S}} (x)-z.
	\end{align*}
	Therefore, the proof is complete.
\end{proof}	

\begin{lemma} \label{lemma:CCSTiFi}
Let $z \in \mathcal{H}$ and 
set $\left( \forall i \in \{1,\ldots, m\}  \right) $ $(\forall x \in \mathcal{H})$ $F_{i}x := G_{i} (x+z) -z$ as well as $\mathcal{S}_{F} := \{F_{1}, \ldots, F_{m}\}$.
	Then the following statements hold:
	\begin{enumerate}
		\item \label{lemma:CCSTiFi:CCSTF} $(\forall x \in \mathcal{H})$ $ \CC{\mathcal{S}} x =  z+ \CC{\mathcal{S}_{F}}(x-z)$.
		\item \label{lemma:CCSTiFi:CCSTF:k} $(\forall x \in \mathcal{H})$ $(\forall k \in \mathbb{N})$ $ \CC{\mathcal{S}}^{k} x =  z+ \CC{\mathcal{S}_{F}}^{k}(x-z)$.
		\item \label{lemma:CCSTiFi:Fix} $\cap^{m}_{i=1}\Fix F_{i} = \cap^{m}_{i=1}\Fix G_{i} -z$.
		\item   \label{lemma:CCSTiFi:Projec:Fix} $(\forall x \in \mathcal{H})$ $\Pro_{  \cap^{m}_{i=1}\Fix G_{i} }x = z + \Pro_{\cap^{m}_{i=1}\Fix F_{i}}(x-z)$.
	\end{enumerate} 
\end{lemma}	

\begin{proof}
	\cref{lemma:CCSTiFi:CCSTF}:  By \cref{def:cir:map} and by \cref{fact:CircumSubaddi}, we obtain that $(\forall x \in \mathcal{H})$,
	$ \CC{\mathcal{S}}x =\CCO{ \left( \left\{ G_{1} x, \ldots, G_{m}x \right\} \right)} = \CCO{ \left( \left\{ z+ F_{1}(x-z), \ldots, z+ F_{m}(x-z) \right\} \right)} =\CCO{\left( \mathcal{S}_{F}(x-z)  \right) }+z =z+ \CC{\mathcal{S}_{F}}(x-z)$. 

	\cref{lemma:CCSTiFi:CCSTF:k}: We prove by induction.  Clearly, the result holds for $k=0$. Assume $(\forall x \in \mathcal{H})$ $\CC{\mathcal{S}}^{k} x =  z+ \CC{\mathcal{S}_{F}}^{k}(x-z)$ holds for some $k \geq 0$. Let $y \in \mathcal{H}$. Now by \cref{lemma:CCSTiFi:CCSTF} above and by  inductive hypothesis,  
	$
	\CC{\mathcal{S}}^{k+1} y=  \CC{\mathcal{S}} ( \CC{\mathcal{S}}^{k} y ) = z +  \CC{\mathcal{S}_{F}} ( \CC{\mathcal{S}}^{k} (y) -z ) = z +  \CC{\mathcal{S}_{F}} (z + \CC{\mathcal{S}_{F}}^{k} (y-z) -z ) =z + \CC{\mathcal{S}_{F}}^{k+1} (y-z). 
	$
	 Hence, we proved \cref{lemma:CCSTiFi:CCSTF:k} by induction.
	 
	\cref{lemma:CCSTiFi:Fix}: This is a direct result from \cref{lemma:TAffineFLinear}\cref{lemma:TAffineFLinear:Fix}.
	
	\cref{lemma:CCSTiFi:Projec:Fix}: This  follows from  \cref{fac:SetChangeProje} and 	\cref{lemma:CCSTiFi:Fix} above.
\end{proof}	

\subsection*{Properties of circumcentered isometry methods}
Recall our global assumptions that
\begin{empheq}[box=\mybluebox]{equation*}
\left(\forall i \in \{1, \ldots, m\} \right) \quad T_{i} : \mathcal{H} \rightarrow \mathcal{H} ~\text{is affine isometry},
\end{empheq}
and 
\begin{empheq}[box=\mybluebox]{equation*}
\mathcal{S} :=\{ T_{1}, \ldots, T_{m} \} \quad \text{with} \quad \cap^{m}_{j=1} \Fix T_{j} \neq \varnothing.
\end{empheq}
From now on, denote by
\begin{empheq}[box=\mybluebox]{equation*}
\Omega ( T_{1}, \ldots, T_{m}) :=  \Big\{ T_{i_{r}}\cdots T_{i_{2}}T_{i_{1}}  ~\Big|~ r \in \mathbb{N}, ~\mbox{and}~ i_{1}, \ldots,  i_{r} \in \{1, \ldots,m\}    \Big\}
\end{empheq}
which is the set consisting of all finite compositions of operators from $\{T_{1}, \ldots, T_{m}\}$.
 We use the empty product convention, so for $r=0$, $T_{i_{0}}\cdots T_{i_{2}}T_{i_{1}} = \Id$.

The following \cref{fact:CCS:proper:NormPres:T}\cref{fact:CCS:proper:NormPres:T:prop} makes the circumcentered method induced by $\mathcal{S}$ defined in \cref{eq:CIMSequence} well-defined. We call the circumcentered method induced by a set of isometries the \emph{circumcentered isometry method} (CIM).

\begin{fact}  \label{fact:CCS:proper:NormPres:T} {\rm \cite[Theorem~3.3, Lemma~3.5 and Proposition~4.2]{BOyW2019Isometry}}
	Let $x \in \mathcal{H}$. Then the following statements hold:
	\begin{enumerate}
		\item \label{fact:CCS:proper:NormPres:T:prop} The circumcenter mapping
		$\CC{\mathcal{S}} : \mathcal{H} \rightarrow \mathcal{H}$ induced by
		$\mathcal{S}$ is proper; moreover, 
		$\CC{\mathcal{S}}x$ is the unique point satisfying the two conditions
		below:
		\begin{enumerate}
			\item  \label{thm:CCS:proper:NormPres:T:i} $\CC{\mathcal{S}}x\in  \aff (\mathcal{S}(x))$, and
			\item  \label{thm:CCS:proper:NormPres:T:ii} $
			\left\{  \norm{\CC{\mathcal{S}}x-Tx } ~|~ T \in \mathcal{S} \right\} $ is a singleton.
		\end{enumerate}
		
		\item \label{fact:CCS:proper:NormPres:T:AllU} Let $W$ be a nonempty closed  convex set of $\cap^{m}_{i=1} \Fix T_{i}$.  
		Then $\CC{\mathcal{S}}x= \Pro_{\aff
			(\mathcal{S}(x))}(\Pro_{\cap^{m}_{j=1} \Fix T_{j}} x) = \Pro_{\aff
			(\mathcal{S}(x))}(\Pro_{W} x)$.
		\item \label{fact:CCS:proper:NormPres:T:F} Let $F: \mathcal{H} \to \mathcal{H}$ satisfy  $(\forall y \in \mathcal{H})$ $F(y) \in \aff (\mathcal{S}(y))$. Then $(\forall z \in \cap^{m}_{i=1} \Fix T_{i}  )$ $\norm{z- \CC{\mathcal{S}}x }^{2} + \norm{\CC{\mathcal{S}}x -Fx}^{2} =\norm{z- Fx}^{2} $.
		\item \label{fact:CCS:proper:NormPres:T:Id} If $\Id \in \aff \mathcal{S}$, then $(\forall z \in \cap^{m}_{i=1} \Fix T_{i})$ $\norm{z- \CC{\mathcal{S}}x }^{2} + \norm{\CC{\mathcal{S}}x -x}^{2} =\norm{z- x}^{2}$.
		\item \label{fact:CCS:proper:NormPres:T:PaffU}   Let $W$ be a nonempty closed  affine subspace of $\cap^{m}_{i=1} \Fix T_{i}$.   Then $(\forall k \in \mathbb{N})$ $\Pro_{W} \CC{\mathcal{S}}^{k}=\Pro_{W} $.
	\end{enumerate}
\end{fact}

\begin{fact} {\rm \cite[Proposition~3.7]{BOyW2019Isometry}}   \label{fact:FixCCS:F1Ft:F} 
 Let $F_{1}, \ldots, F_{t}$ be isometries from $\mathcal{H}$ to $\mathcal{H}$. Let $\mathcal{S}$ be a finite subset of $\Omega(F_{1}, \ldots,F_{t})$.
		Let $\{ \Id, F_{1}, F_{2}F_{1}, \ldots, F_{t} F_{t-1} \cdots F_{2}F_{1} \} \subseteq \mathcal{S}$. Then $\Fix \CC{\mathcal{S}} = \cap^{t}_{j=1} \Fix F_{j}  = \cap^{m}_{i=1} \Fix G_{i} $.
\end{fact}

The following result is a generalization of \cite[Proposition~3.8]{BOyW2019Isometry}.
\begin{proposition} \label{prop:CCS:FQNE}
	The following statements hold:
	\begin{enumerate}
		\item \label{prop:CCS:FQNE:general} Assume that $\Id \in \aff \mathcal{S}$ and that $\Fix \CC{\mathcal{S}} \subseteq \cup^{m}_{i=1} \Fix T_{i}$.	Then $\CC{\mathcal{S}}$ is firmly quasinonexpansive.
		\item \label{prop:CCS:FQNE:Id} If $\Id \in \mathcal{S}$, then 	$\CC{\mathcal{S}}$ is firmly quasinonexpansive.
	\end{enumerate}
\end{proposition}

\begin{proof}
	\cref{prop:CCS:FQNE:general}: By assumptions and by  \cref{fact:FixCCS:F1Ft:T}\cref{fact:CW:FixPointSet:equa} and \cref{fact:CCS:proper:NormPres:T}\cref{fact:CCS:proper:NormPres:T:Id}, we obtain that $\Fix \CC{\mathcal{S}} = \cap^{m}_{i=1} \Fix T_{i}$ and that
	\begin{align*}
		(\forall x \in \mathcal{H})  (\forall y \in \cap^{m}_{i=1} \Fix T_{i}) \quad \norm{y- \CC{\mathcal{S}}x }^{2} + \norm{\CC{\mathcal{S}}x -x}^{2} =\norm{y- x}^{2}.
	\end{align*}
	Hence,
		\begin{align*}
	(\forall x \in \mathcal{H})  (\forall y \in \Fix \CC{\mathcal{S}}) \quad \norm{\CC{\mathcal{S}}x-y}^{2} + \norm{\CC{\mathcal{S}}x-x}^{2} \leq \norm{x-y}^{2},
	\end{align*}
	which, by \cref{defn:Nonexpansive}\cref{FirmlyQuasiNonex}, means that $\CC{\mathcal{S}}$ is firmly quasinonexpansive.
	
	\cref{prop:CCS:FQNE:Id}: Clearly,  $\Id \in \mathcal{S}$ implies that $\Fix \CC{\mathcal{S}} \subseteq \mathcal{H} =\Fix \Id =\cup^{m}_{i=1} \Fix T_{i}$ and that  $\Id \in \aff \mathcal{S}$.
	Hence, \cref{prop:CCS:FQNE:Id} comes from \cref{prop:CCS:FQNE:general}.
\end{proof}

\begin{proposition} \label{prop:separation:CCS}
	Let $p \in \mathbb{N} \smallsetminus \{0\}$. Denote by $\I:= \{1,\ldots, m \}$ and $\J:= \{1,\ldots, p \}$ . Let $(\forall j \in \J)$ $\mathcal{S}_{j} \subseteq  \{T_{1}, \ldots, T_{m} \}$ such that $\Id \in \mathcal{S}_{j}$ and $\cup^{p}_{t=1} \mathcal{S}_{t} = \{T_{1}, \ldots, T_{m} \}$. Then the following hold:
	\begin{enumerate}
		\item \label{prop:separation:CCS:Fix} $\cap^{p}_{j=1} \Fix \CC{\mathcal{S}_{j}} =\cap^{m}_{i=1} \Fix T_{i}$.
		\item \label{prop:separation:CCS:compose} $\CC{\mathcal{S}_{p}} \cdots \CC{\mathcal{S}_{1}} $ is strictly quasinonexpansive and $ \Fix \CC{\mathcal{S}_{p}} \cdots \CC{\mathcal{S}_{1}}  = \cap^{m}_{i=1} \Fix T_{i}$.
		\item \label{prop:separation:CCS:combination} Let $(\omega_{j})_{j\in \J}$ be real numbers in $\left]0,1\right]$ such that $\sum_{j \in \J} \omega_{j} =1$. Then $\sum_{j \in \J} \omega_{j}\CC{\mathcal{S}_{j}} $ is firmly quasinonexpansive. Moreover, $\Fix \sum_{j \in \J} \omega_{j}\CC{\mathcal{S}_{j}} =\cap^{m}_{i=1} \Fix T_{i}$.
	\end{enumerate}
\end{proposition}

\begin{proof}
	\cref{prop:separation:CCS:Fix}: Because $(\forall j \in \J)$ $\Id \in \mathcal{S}_{j} \subseteq  \{T_{1}, \ldots, T_{m} \}$, by	 \cref{fact:FixCCS:F1Ft:T}\cref{fact:Id:FixPointSet:equa}, $(\forall j \in \J)$  $\cap^{m}_{i=1} \Fix T_{i} \subseteq \cap_{T \in \mathcal{S}_{j}} \Fix T = \Fix \CC{\mathcal{S}_{j}}$. Hence,
	\begin{align}  \label{eq:prop:separation:CCS:Fix:<}
	\cap^{m}_{i=1} \Fix T_{i} \subseteq \cap^{p}_{j=1} \Fix \CC{\mathcal{S}_{j}}.
	\end{align}
	On the other hand, because $\cup^{p}_{t=1} \mathcal{S}_{t} = \{T_{1}, \ldots, T_{m} \}$, for every $i \in \I$  there exists $t_{i} \in \J$ such that $T_{i} \in \mathcal{S}_{t_{i}}$. By the assumption, $\Id \in \mathcal{S}_{t_{i}}$, and by \cref{fact:FixCCS:F1Ft:T}\cref{fact:Id:FixPointSet:equa} again, $\Fix \CC{\mathcal{S}_{t_{i}}} =\cap_{T \in \mathcal{S}_{t_{i}}  } \Fix T \subseteq \Fix T_{i}$, which implies that $\cap^{p}_{j=1} \Fix \CC{\mathcal{S}_{j}} \subseteq \Fix \CC{\mathcal{S}_{t_{i}}} \subseteq \Fix T_{i}$. Moreover, because  the above $i \in \I$ is chosen arbitrarily, we have 
	\begin{align}  \label{eq:prop:separation:CCS:Fix:>}
	 \cap^{p}_{j=1} \Fix \CC{\mathcal{S}_{j}}  \subseteq \cap^{m}_{i=1} \Fix T_{i}.
	\end{align}
	Therefore, \cref{eq:prop:separation:CCS:Fix:<} and \cref{eq:prop:separation:CCS:Fix:>} yield \cref{prop:separation:CCS:Fix}.
	
	\cref{prop:separation:CCS:compose}: Let $j \in \J$. By assumption,  $\Id \in \mathcal{S}_{j}$, and by \cref{prop:CCS:FQNE}\cref{prop:CCS:FQNE:Id},    $\CC{\mathcal{S}_{j}} $ is firmly quasinonexpansive. By \cref{rem:NonexpImplication}, we know that   $\CC{\mathcal{S}_{j}} $ is strictly quasinonexpansive. In addition, the global assumption  $\cap^{m}_{i=1} \Fix T_{i} \neq \varnothing$ and the assumption, $\mathcal{S}_{j} \subseteq  \{T_{1}, \ldots, T_{m} \}$, imply  that  $\cap_{T \in \mathcal{S}_{j}} \Fix T \neq \varnothing$. Hence, by \cite[Corollary~4.50]{BC2017}, $\CC{\mathcal{S}_{p}} \cdots \CC{\mathcal{S}_{1}} $ is strictly quasinonexpansive and $ \Fix \CC{\mathcal{S}_{p}} \cdots \CC{\mathcal{S}_{1}}  = \cap^{p}_{j=1} \Fix \CC{\mathcal{S}_{j}} $. Combine  the identity with the \cref{prop:separation:CCS:Fix} above to deduce $ \Fix \CC{\mathcal{S}_{p}} \cdots \CC{\mathcal{S}_{1}}  = \cap^{m}_{i=1} \Fix T_{i}$.
	
	\cref{prop:separation:CCS:combination}:  Let $j \in \J$. By assumption,  $\Id \in \mathcal{S}_{j}$, and by \cref{prop:CCS:FQNE}\cref{prop:CCS:FQNE:Id},    $\CC{\mathcal{S}_{j}} $ is firmly quasinonexpansive. By \cite[Corollary~4.48]{BC2017},  $\sum_{j \in \J} \omega_{j}\CC{\mathcal{S}_{j}} $ is firmly quasinonexpansive.
	
	In addition, by \cref{rem:NonexpImplication}, for every $j \in \J$, $\CC{\mathcal{S}_{j}} $ is  firmly quasinonexpansive implies that $\CC{\mathcal{S}_{j}} $ is  quasinonexpansive. By  \cite[Proposition~4.47]{BC2017} and the \cref{prop:separation:CCS:Fix} above, we obtain that  $\Fix \sum_{j \in \J} \omega_{j}\CC{\mathcal{S}_{j}} =\cap^{p}_{j=1} \Fix \CC{\mathcal{S}_{j}}=\cap^{m}_{i=1} \Fix T_{i}$.
\end{proof}
\begin{lemma} \label{lem:TiT:CapFixPerp}
	Suppose that $T_{1}, \ldots, T_{m}$ are  linear. Then $(\forall x \in \mathcal{H})$ $\left(  \forall i \in \{1, \ldots, m\} \right)$ $T_{i}x -x \in (\cap^{m}_{j=1} \Fix T_{j})^{\perp}$.
\end{lemma}

\begin{proof}	
	Let $x \in \mathcal{H}$ and let $i \in \{1, \ldots, m\}$. 
	By \cref{lemma:FixRangeNonexpansive}, $T_{i}x -x \in \Range (T_{i} - \Id) \subseteq \overline{\Range}\, (\Id -T_{i}   ) =\left( \Fix T_{i} \right)^{\perp} \subseteq (\cap^{m}_{j=1} \Fix T_{j})^{\perp}$.
\end{proof}

\begin{lemma} \label{lem:CCSkCAPperp}
	Suppose that $T_{1}, \ldots, T_{m}$ are linear and that $T_{1}=\Id$.  Then the following statements hold:
	\begin{enumerate}
		\item \label{lem:CCSkCAPperp:CCSxx} $(\forall x \in \mathcal{H})$ $\CC{\mathcal{S}} x -x  \in (\cap^{m}_{i=1} \Fix T_{i})^{\perp}$.
		\item \label{lem:CCSkCAPperp:CCSk} $(\forall x \in  (\cap^{m}_{i=1} \Fix T_{i})^{\perp} )$ $ (\forall k \in \mathbb{N}) $ $\CC{\mathcal{S}}^{k} x \in (\cap^{m}_{i=1} \Fix T_{i})^{\perp}$.
	\end{enumerate}
\end{lemma}

\begin{proof}
	\cref{lem:CCSkCAPperp:CCSxx}: Let $x \in \mathcal{H}$. Since $\Id \in \mathcal{S}$, by  \cref{fact:CCS:proper:NormPres:T}\cref{fact:CCS:proper:NormPres:T:prop} and \cref{fact:CW:Welldefined:Formula}, there exist $\alpha_{1}, \ldots,\alpha_{m-1} $ in $\mathbb{R}$
	such that 
	\begin{align*}
	\CC{\mathcal{S}}x -x= \sum^{m-1}_{i=1} \alpha_{i} (T_{i+1}x -x).
	\end{align*}
	On the other hand, by \cref{lem:TiT:CapFixPerp}, we have that $(\forall i \in \{1, \ldots, m-1\})$ $T_{i+1}x -x \in  (\cap^{m}_{j=1} \Fix T_{j})^{\perp}$. Since $(\cap^{m}_{j=1} \Fix T_{j})^{\perp}$ is a linear subspace,   \cref{lem:CCSkCAPperp:CCSxx} is true.
	
	\cref{lem:CCSkCAPperp:CCSk}: For every $x \in  (\cap^{m}_{i=1} \Fix T_{i})^{\perp}$, by \cref{lem:CCSkCAPperp:CCSxx} above, $\CC{\mathcal{S}}x =x + \left(\CC{\mathcal{S}}x -x \right) \in (\cap^{m}_{i=1} \Fix T_{i})^{\perp}$. Therefore, the required result follows from \cref{lem:CCSkCAPperp:CCSxx} inductively.   
\end{proof}	

\begin{remark}
	\begin{enumerate}
		\item  In view of  \cref{fact:FixCCS:F1Ft:F}, we note that \cref{lem:TiT:CapFixPerp} and  \cref{lem:CCSkCAPperp} reduce to \cite[Propositions~5.4 and 5.5]{BOyW2019Isometry} respectively when the related isometries are reflectors.
		\item \cref{lem:CCSkCAPperp}\cref{lem:CCSkCAPperp:CCSk} implies that when we use the CIM, $(\CC{\mathcal{S}}^{k} x)_{k \in \mathbb{N}}$, to  find the best approximation $\Pro_{  \cap^{m}_{i=1}\Fix T_{i} }x $, if we choose our initial point $x$ in the linear subspace $ (\cap^{m}_{i=1} \Fix T_{i})^{\perp} $ and if $\Pro_{  \cap^{m}_{i=1}\Fix T_{i} }x \neq 0$, then it is impossible for us to find the $\Pro_{  \cap^{m}_{i=1}\Fix T_{i} }x $ in finitely many steps.
		This is consistent with \cite[Section~4]{BCS2019} which shows that to satisfy one step convergence of CRM for hyperplane intersection, there are certain requirements for the initial points.
	\end{enumerate}
	
\end{remark}

The following result reduces to \cite[Proposition~5.3]{BOyW2019Isometry} when the related isometries are reflectors.
\begin{theorem} \label{theo:PARALLIsometryLinear}
	Let $z \in \mathcal{H}$.  Set $\left( \forall i \in \{1,\ldots, m\}  \right) $ $(\forall x \in \mathcal{H})$ $F_{i}x := T_{i} (x+z) -z$ and $\mathcal{S}_{F} = \{F_{1}, \ldots, F_{m}\}$. Let $\gamma \in \left[0,1\right[\,$. Then for every $x\in \mathcal{H}$, the following statements are equivalent:
	\begin{enumerate}
		\item \label{item:prop:PARALLIsometryLinear:T} $(\forall k \in \mathbb{N}) \quad \norm{ \CC{\mathcal{S}}^{k}x  - \Pro_{  \cap^{m}_{i=1}\Fix T_{i} }x} \leq \gamma^{k} \norm{x - \Pro_{  \cap^{m}_{i=1}\Fix T_{i} }x}$.
		\item \label{item:prop:PARALLIsometryLinear:F} $(\forall k \in \mathbb{N}) \quad \norm{ \CC{\mathcal{S}_{F}}^{k}(x-z) - \Pro_{\cap^{m}_{i=1}\Fix F_{i}}(x-z)} \leq \gamma^{k} \norm{(x - z) - \Pro_{\cap^{m}_{i=1}\Fix F_{i}}(x-z)}$.
	\end{enumerate}
Consequently, the following assertions hold:
\begin{enumerate}[label=(\alph*)]
	\item \label{item:prop:PARALLIsometryLinear:a} Given $x \in \mathcal{H}$, $(\CC{\mathcal{S}}^{k}x)_{k \in \mathbb{N}}$ converges linearly to $\Pro_{ \cap^{m}_{i=1} T_{i}}x$ with linear rate $\gamma$ if and only if $(\CC{\mathcal{S}_{F}}^{k} (x -z))_{k \in \mathbb{N}}$ converges linearly to $\Pro_{ \cap^{m}_{i=1} F_{i}} (x-z)$ with linear rate $\gamma$. 
	\item \label{item:prop:PARALLIsometryLinear:b}  $(\forall x \in \mathcal{H})$ $(\CC{\mathcal{S}}^{k}x)_{k \in \mathbb{N}}$ converges linearly to $\Pro_{ \cap^{m}_{i=1} T_{i}}x$ with linear rate $\gamma$ if and only if $(\forall y \in \mathcal{H})$ $(\CC{\mathcal{S}_{F}}^{k} y)_{k \in \mathbb{N}}$ converges linearly to $\Pro_{ \cap^{m}_{i=1} F_{i}}y$ with linear rate $\gamma$. 
\end{enumerate}
\end{theorem}

\begin{proof}
	By \cref{lemma:CCSTiFi}\cref{lemma:CCSTiFi:Fix}, $\cap^{m}_{i=1} T_{i} \neq \varnothing$ is equivalent to $\cap^{m}_{i=1} F_{i} \neq \varnothing$. Hence,
	\cref{fact:CCS:proper:NormPres:T}\cref{fact:CCS:proper:NormPres:T:prop} and  \cref{lemma:TAffineFLinear}\cref{lemma:TAffineFLinear:Isometric} yield  that for every  $x \in \mathcal{H}$, both $(\CC{\mathcal{S}}^{k}x)_{k \in \mathbb{N}}$ and $(\CC{\mathcal{S}_{F}}^{k} x)_{k \in \mathbb{N}}$ are well-defined.

By \cref{lemma:CCSTiFi}\cref{lemma:CCSTiFi:Projec:Fix}$\&$\cref{lemma:CCSTiFi:CCSTF:k}, for every $x \in \mathcal{H}$, 
\begin{align*}
x - \Pro_{  \cap^{m}_{i=1}\Fix T_{i} }x =  (x - z) - \Pro_{\cap^{m}_{i=1}\Fix F_{i}}(x-z),
\end{align*}
and 
\begin{align*}
  (\forall k \in \mathbb{N}) \quad \CC{\mathcal{S}}^{k}x  - \Pro_{  \cap^{m}_{i=1}\Fix T_{i} }x = \CC{\mathcal{S}_{F}}^{k}(x-z) - \Pro_{\cap^{m}_{i=1}\Fix F_{i}}(x-z).
\end{align*}
Therefore, we obtain that \cref{item:prop:PARALLIsometryLinear:T}  $\Leftrightarrow$ \cref{item:prop:PARALLIsometryLinear:F}.
Moreover, it is clear that both \cref{item:prop:PARALLIsometryLinear:a} and \cref{item:prop:PARALLIsometryLinear:b} follow easily from the equivalence of \cref{item:prop:PARALLIsometryLinear:T} and \cref{item:prop:PARALLIsometryLinear:F}.
The proof is complete.
\end{proof}

\begin{remark}
	\cref{theo:PARALLIsometryLinear}, \cref{prop:IsometryAffine} and \cref{lemma:TAffineFLinear}\cref{lemma:TAffineFLinear:TF} allow us to assume that 
	 all of the associated isometries are linear when we study the linear convergence of CIMs.
\end{remark}

\section{Linear convergence of circumcentered isometry methods} \label{section:LinearConverg}
The linear convergence results in this section hinge on the following two facts. 
\begin{fact}  \label{fact:CWP:line:conv} {\rm  \cite[Theorem~4.14]{BOyW2019Isometry}}
	Recall that $(\forall i \in \{1,\ldots,m\})$ $T_{i} : \mathcal{H} \rightarrow \mathcal{H} $ is an affine isometry  with $ \cap^{m}_{j=1} \Fix T_{j} \neq \varnothing$ and that $\mathcal{S} :=\{ T_{1}, \ldots,  T_{m} \}$. Let $W$ be a nonempty closed   affine subspace of $\cap^{m}_{j=1} \Fix T_{j} $. 	
	Assume that there exist $F: \mathcal{H} \to \mathcal{H}$ and $\gamma \in \left[0,1\right[$ such that $(\forall x \in \mathcal{H})$ $F(x) \in \aff(\mathcal{S}(x) )$ and
	$(\forall x \in \mathcal{H})$ $ \norm{Fx-\Pro_{W} x} \leq \gamma \norm{x -\Pro_{W} x}.
	$
	Then
	\begin{align*}  
	(\forall x \in \mathcal{H}) (\forall k \in \mathbb{N}) \quad \norm{\CC{\mathcal{S}}^{k}x-\Pro_{W} x} \leq \gamma^{k} \norm{x -\Pro_{W} x}.
	\end{align*}
\end{fact}

\begin{fact}  {\rm  \cite[Theorem~4.15]{BOyW2019Isometry}}
	\label{fact:CCSLinearConvTSFirmNone}
		Suppose that $\mathcal{H}= \mathbb{R}^{n}$. Recall that $(\forall i \in  \{1,\ldots,m\})$ $T_{i} : \mathcal{H} \rightarrow \mathcal{H} $ is an affine isometry  with $ \cap^{m}_{j=1} \Fix T_{j}  \neq \varnothing$ and that $\mathcal{S} :=\{ T_{1}, \ldots,  T_{m} \}$.  Let $T_{\mathcal{S}} \in \aff \mathcal{S}$  satisfy that $ \Fix T_{\mathcal{S}} \subseteq \cap_{T \in \mathcal{S}} \Fix T$.  Assume that $T_{\mathcal{S}} $ is linear and $\alpha$-averaged with $\alpha \in \left]0,1\right[\,$. Then $\norm{T_{\mathcal{S}} \Pro_{(\cap_{T \in \mathcal{S}} \Fix T)^{\perp} }} \in \left[0,1\right[\,$. Moreover, 
		\begin{align*}
		(\forall x \in \mathcal{H})  (\forall k \in \mathbb{N}) \quad \norm{\CC{\mathcal{S}}^{k}x-\Pro_{\cap_{T \in \mathcal{S}} \Fix T} x} \leq \norm{T_{\mathcal{S}} \Pro_{(\cap_{T \in \mathcal{S}} \Fix T)^{\perp} }}^{k} \norm{x -\Pro_{\cap_{T \in \mathcal{S}} \Fix T } x}.
		\end{align*}
\end{fact}
Note that because $T_{\mathcal{S}} $ is linear, $0 \in \Fix T_{\mathcal{S}} \subseteq \cap_{T \in \mathcal{S}} \Fix T $, which implies that $(\forall T \in \mathcal{S})$,  $T$ must be linear.
In addition, actually, $T_{\mathcal{S}} \in \aff \mathcal{S}$ and $ \Fix T_{\mathcal{S}} \subseteq \cap_{T \in \mathcal{S}} \Fix T$ imply that $\Fix T_{\mathcal{S}}  = \cap_{T \in \mathcal{S}} \Fix T$.

\subsection*{Linear convergence of CIMs in finite-dimensional spaces} \label{sec:LCCIMRn}

\begin{lemma} \label{lemma:Propo:T:SUM}
	Let $t \in \mathbb{N} \smallsetminus \{0\}$ and let $\I := \{1,2,\ldots, t\}$. Let $F_{1}, F_{2},  \ldots, F_{t}$ be nonexpansive and linear on $\mathcal{H}$. Let $(\omega_{i})_{i \in \I}$ be real numbers in $\left]0,1\right]$ such that $\sum_{i \in \I} \omega_{i} =1$ and let $(\alpha_{i})_{i \in \I}$ be real numbers in $\left]0,1\right[\,$.  Denote
	\begin{align*}
	A := \sum_{i \in \I} \omega_{i} A_{i} \quad \text{where} \quad (\forall i \in \I) \quad A_{i} :=
	(1-\alpha_{i}) \Id  + \alpha_{i}   F_{i}.
	\end{align*}
	Then the following assertions hold:
	\begin{enumerate}
		\item  \label{lemma:Propo:T:SUM:T:Averaged} Let $\alpha := \sum_{i \in \I} \omega_{i} \alpha_{i}$.  Then $A$ is $\alpha$-averaged and linear.
		\item  \label{lemma:Propo:T:SUM:FixT}  $\Fix A =\cap_{i \in \I} \Fix F_{i}$.
		\item  \label{lemma:Propo:T:SUM:IncluAffi}  Assume that $\widetilde{\mathcal{S}}$ is a finite set of operators such that $\{\Id, F_{1}, F_{2},  \ldots, F_{t} \}  \subseteq \widetilde{\mathcal{S}}$.  Then $A \in \aff \widetilde{\mathcal{S}}$. 
	\end{enumerate}
\end{lemma}

\begin{proof}
	\cref{lemma:Propo:T:SUM:T:Averaged}: Because $F_{1}, F_{2},  \ldots, F_{t}$ are linear, $A$ is linear.  Since $F_{1}, F_{2},  \ldots, F_{t}$ are nonexpansive, $(\forall i \in \I)$ $A_{i}$ is $\alpha_{i}$-averaged. Hence, the required result follows from	\cref{fact:T:Avera:Sum:FixT}.
	
	\cref{lemma:Propo:T:SUM:FixT}: The result follows from \cref{fact:FixSumInters}.

	\cref{lemma:Propo:T:SUM:IncluAffi}:  By definition, $(\forall i \in \I)$ $A_{i} \in \aff \{ \Id,  F_{1}, F_{2},  \ldots, F_{t} \}$. Hence, $A \in \aff \{ \Id,  F_{1}, F_{2},  \ldots, F_{t} \}  \subseteq \aff \widetilde{\mathcal{S}}$.
\end{proof}	
The following result reduces to \cite[Proposition~5.15]{BOyW2019Isometry}
when the isometries are reflectors.
\begin{theorem} \label{theo:CCS:LineaConve:F1Fn}
	Suppose that $\mathcal{H} = \mathbb{R}^{n}$. Let $F_{1}, F_{2},  \ldots, F_{t}$ be linear isometries on $\mathcal{H}$. Assume  that $\widetilde{\mathcal{S}}$ is a finite subset of $\Omega(F_{1}, \ldots, F_{t})$, where $\Omega(F_{1}, \ldots, F_{t})$ consists of all 
	finite compositions of operators from $\{ F_{1}, \ldots, F_{t} \}$. Assume that   $\{ \Id,  F_{1}, F_{2},  \ldots, F_{t} \}  \subseteq \widetilde{\mathcal{S}}$.  Let $(\omega_{i})_{i \in \I}$ be real numbers in $\left]0,1\right]$ such that $\sum_{i \in \I} \omega_{i} =1$ and let $(\alpha_{i})_{i \in \I}$ be real numbers in $\left]0,1\right[\,$.  Denote $A := \sum^{t}_{i =1} \omega_{i} A_{i}$ where $(\forall i \in \{1,\ldots,t\})$ $A_{i} :=
	(1-\alpha_{i}) \Id  +  \alpha_{i}  F_{i}$.  Then the following statements hold:
	\begin{enumerate}
		\item \label{prop:CCS:LineaConve:F1Fn:Fix} $\Fix \CC{\widetilde{\mathcal{S}}}= \cap_{T \in \widetilde{\mathcal{S}}} \Fix T=  \cap^{t}_{i=1} \Fix F_{i}  = \Fix A $.
		\item \label{prop:CCS:LineaConve:F1Fn:LineaConv} $\norm{A \Pro_{(\cap^{t}_{i=1} \Fix F_{i})^{\perp}} } <1$. Moreover, 
		\begin{align*}
		(\forall x \in \mathcal{H})  (\forall k \in \mathbb{N}) \quad \norm{\CC{\widetilde{\mathcal{S}}}^{k}x - \Pro_{\cap^{t}_{i=1} \Fix F_{i}}x} \leq \norm{A \Pro_{(\cap^{t}_{i=1} \Fix F_{i})^{\perp}} }^{k} \norm{x-\Pro_{\cap^{t}_{i=1} \Fix F_{i}} x }.
		\end{align*}
		Consequently, $(\CC{\widetilde{\mathcal{S}}}^{k}x)_{k \in \mathbb{N}}$ converges to $\Pro_{\cap^{t}_{i=1} \Fix F_{i}}x $ with a linear rate $\norm{A \Pro_{(\cap^{t}_{i=1} \Fix F_{i})^{\perp}} } $.
	\end{enumerate}
\end{theorem}	

\begin{proof}
	By assumptions and by	\cref{fact:CCS:proper:NormPres:T}\cref{fact:CCS:proper:NormPres:T:prop}, $\CC{\widetilde{\mathcal{S}}}$ is proper. 
	
	\cref{prop:CCS:LineaConve:F1Fn:Fix}: By assumption, $\widetilde{\mathcal{S}} \subseteq \Omega(F_{1}, \ldots, F_{t})$, so every operator in  $\widetilde{\mathcal{S}} $ is a finite composition of operators 
from $\{ F_{1}, \ldots, F_{t} \}$. Hence, $  \cap^{t}_{i=1} \Fix F_{i}  \subseteq \cap_{T \in \widetilde{\mathcal{S}}} \Fix T $.
	Moreover, because $\{ \Id,  F_{1}, F_{2},  \ldots, F_{t} \}  \subseteq \widetilde{\mathcal{S}}$,  $\cap_{T \in \widetilde{\mathcal{S}}} \Fix T \subseteq \cap^{t}_{i=1} \Fix F_{i} $. Hence,  $  \cap^{t}_{i=1} \Fix F_{i}  =\cap_{T \in \widetilde{\mathcal{S}}} \Fix T $.
	By \cref{lemma:Propo:T:SUM}\cref{lemma:Propo:T:SUM:FixT}, $\Fix A = \cap^{t}_{i=1} \Fix F_{i}$.
	By  \cref{fact:FixCCS:F1Ft:T}\cref{fact:Id:FixPointSet:equa},  we have $\Fix \CC{\widetilde{\mathcal{S}}}= \cap_{T \in \widetilde{\mathcal{S}}} \Fix T$. Hence, $\Fix \CC{\widetilde{\mathcal{S}}}= \cap_{T \in \widetilde{\mathcal{S}}} \Fix T=  \cap^{t}_{i=1} \Fix F_{i}  = \Fix A $.
	
	\cref{prop:CCS:LineaConve:F1Fn:LineaConv}: This follows from \cref{prop:CCS:LineaConve:F1Fn:Fix} above,  \cref{lemma:Propo:T:SUM}\cref{lemma:Propo:T:SUM:T:Averaged}$\&$\cref{lemma:Propo:T:SUM:IncluAffi}, \cref{prop:TFNorm}\cref{prop:TFNorm:TTperp},   and \cref{fact:CCSLinearConvTSFirmNone}.
\end{proof}

\begin{lemma} \label{lemma:Propo:T:SUM:Produ}
	Let $t \in \mathbb{N} \smallsetminus \{0\}$, let $\I := \{1, \ldots, t\}$, let $F_{1}, F_{2},  \ldots, F_{t}$ be nonexpansive and linear, let $(\omega_{i})_{i \in \I}$ be real numbers in $\left]0,1\right]$ s.t. $\sum_{i \in \I} \omega_{i} =1$ and let $(\alpha_{i})_{i \in \I}$ and $(\lambda_{i})_{i \in \I}$ be real numbers in $\left]0,1\right[\,$. Denote
	$ A := \sum_{i \in \I} \omega_{i} A_{i}$ where $A_{1}  := (1-\alpha_{1})  \Id  + \alpha_{1} F_{1} $, $(\forall i \in \I \smallsetminus \{1\})$ $A_{i} :=(1-\alpha_{i})  \Id  + \alpha_{i} \left ( (1-\lambda_{i}) \Id +  \lambda_{i}F_{i} \right )F_{i-1}\cdots F_{1}$.
	Then the following assertions hold:
	\begin{enumerate}
		\item  \label{lemma:Propo:T:SUM:T:Averaged:Produ} Let $\alpha := \sum_{i \in \I} \omega_{i} \alpha_{i}$.  Then $A$ is $\alpha$-averaged and linear.
		\item  \label{lemma:Propo:T:SUM:FixT:Produ}  $\Fix A =\cap^{t}_{i=1} \Fix F_{i}$.
		\item  \label{lemma:Propo:T:SUM:IncluAffi:Produ}  Assume that $\widetilde{\mathcal{S}}$ is a finite set of operators such that $\{ \Id,  F_{1}, F_{2}F_{1},  \ldots, F_{t}\cdots F_{2}F_{1} \} \subseteq \widetilde{\mathcal{S}}$.  Then $A \in \aff \widetilde{\mathcal{S}}$. 
	\end{enumerate}
\end{lemma}

\begin{proof}
	\cref{lemma:Propo:T:SUM:T:Averaged:Produ}: Since $F_{1}, F_{2},  \ldots, F_{t}$ are linear, so is $A$.  
	Since $F_{1}, F_{2},  \ldots, F_{t}$ are nonexpansive, thus $(\forall i \in \I)$ $A_{i}$ is $\alpha_{i}$-averaged. Hence, the required result follows from	\cref{fact:T:Avera:Sum:FixT}.
	
	\cref{lemma:Propo:T:SUM:FixT:Produ}: 
	Since every averaged operator is strictly quasinonexpansive, we have $(\forall i \in \I) $ $(\forall \lambda \in \left]0,1\right[\,)$ $  (1-\lambda) \Id +  \lambda F_{i}$ is strictly quasinonexpansive. Hence, 
	 the result follows from \Cref{fact:FixSumInters,fact:FixProdInters}, since $\cap^{t}_{i=1} \Fix F_{i} \neq \varnothing$.	
	
	\cref{lemma:Propo:T:SUM:IncluAffi:Produ}: Clearly, $A_{1} : =(1-\alpha_{1})   \Id  +  \alpha_{1} F_{1}  \in \aff \{ \Id,  F_{1}, F_{2}F_{1},  \ldots, F_{m}\cdots F_{2}F_{1} \} \subseteq \aff \widetilde{\mathcal{S}}$. Moreover, for every $ i \in \I \smallsetminus \{1\}$,
	\begin{align*}
	A_{i} & =(1-\alpha_{i})   \Id  +  \alpha_{i}  \left (  (1-\lambda_{i})\Id +   \lambda_{i}F_{i} \right )F_{i-1}\cdots F_{1} \\
	& =  (1-\alpha_{i})   \Id  + \alpha_{i}   \left ( (1-\lambda_{i}) F_{i-1}\cdots F_{1} + \lambda_{i} F_{i}F_{i-1}\cdots F_{1} \right ) \\
	&=(1-\alpha_{i})  \Id + \alpha_{i}  (1-\lambda_{i})  F_{i-1}\cdots F_{1} + \alpha_{i} \lambda_{i} F_{i}F_{i-1}\cdots F_{1}\\
	& \in \aff \{ \Id,  F_{1}, F_{2}F_{1},  \ldots, F_{m}\cdots F_{2}F_{1} \} \subseteq \aff \widetilde{\mathcal{S}}.
	\end{align*}
	Hence, $A= \sum_{i \in \I} \omega_{i} A_{i} \in \aff \widetilde{\mathcal{S}}$.
\end{proof}

 The following results is a generalization of \cite[Theorem~3.3]{BCS2018} and \cite[Proposition~5.10]{BOyW2019Isometry}.
\begin{theorem} \label{theo:CCS:LineaConve:F1t} 
	Suppose that  $\mathcal{H} = \mathbb{R}^{n}$. Let $F_{1}, F_{2},  \ldots, F_{t}$ be linear isometries. Assume  that $\widetilde{\mathcal{S}}$ is a finite subset of $\Omega(F_{1}, \ldots, F_{t})$, where $\Omega(F_{1}, \ldots, F_{t})$ consists of all 
	finite compositions of operators from $\{ F_{1}, \ldots, F_{t} \}$. 
	 Assume that  $ \{ \Id,  F_{1}, F_{2}F_{1},  \ldots, F_{t}\cdots F_{2}F_{1} \} \subseteq \widetilde{\mathcal{S}} $. Let $(\omega_{i})_{i \in \I}$ be real numbers in $\left]0,1\right]$ such that $\sum_{i \in \I} \omega_{i} =1$ and let $(\alpha_{i})_{i \in \I}$ and $(\lambda_{i})_{i \in \I}$ be real numbers in $\left]0,1\right[\,$.  Set
	$ A := \sum_{i \in \I} \omega_{i} A_{i}$
	where $A_{1} : =  (1-\alpha_{1}) \Id  + \alpha_{1} F_{1} $ and  $(\forall i \in \I \smallsetminus \{1\})$ $A_{i} := (1-\alpha_{i}) \Id  + \alpha_{i}  \left ( (1-\lambda_{i}) \Id + \lambda_{i}F_{i} \right )F_{i-1}\cdots F_{1}$. 
	Then the following assertions hold:
	\begin{enumerate}
		\item \label{theo:CCS:LineaConve:F1t:Fix} $\Fix \CC{\widetilde{\mathcal{S}}}= \cap_{T \in \widetilde{\mathcal{S}}} \Fix T=  \cap^{t}_{i=1} \Fix F_{i}  = \Fix A $.
		\item \label{theo:CCS:LineaConve:F1t:LineaConv} $\norm{A \Pro_{(\cap^{t}_{i=1} \Fix F_{i})^{\perp}} } \in \left[0,1\right[\,$. Moreover, 
		\begin{align*}
	(\forall x \in \mathcal{H}) (\forall k \in \mathbb{N}) \quad \norm{\CC{\widetilde{\mathcal{S}}}^{k}x - \Pro_{\cap^{t}_{i=1} \Fix F_{i}}x} \leq \norm{A \Pro_{(\cap^{t}_{i=1} \Fix F_{i})^{\perp}} }^{k} \norm{x-\Pro_{\cap^{t}_{i=1} \Fix F_{i}} x }.
		\end{align*}
		Consequently, $(\CC{\widetilde{\mathcal{S}}}^{k}x)_{k \in \mathbb{N}}$ converges to $\Pro_{\cap^{t}_{i=1} \Fix F_{i}}x $ with a linear rate $\norm{A \Pro_{(\cap^{t}_{i=1} \Fix F_{i})^{\perp}} }$.
	\end{enumerate}
\end{theorem}	

\begin{proof}
	By assumptions and by	\cref{fact:CCS:proper:NormPres:T}\cref{fact:CCS:proper:NormPres:T:prop}, $\CC{\widetilde{\mathcal{S}}}$ is proper. 
	
	\cref{theo:CCS:LineaConve:F1t:Fix}:	Because $\widetilde{\mathcal{S}}$ is a finite subset of $\Omega(F_{1}, \ldots, F_{t})$ such that  $ \{ \Id,  F_{1}, F_{2}F_{1},  \ldots, F_{t}\cdots F_{2}F_{1} \} \subseteq \widetilde{\mathcal{S}} $, by  \cref{fact:FixCCS:F1Ft:F},  $\Fix \CC{\widetilde{\mathcal{S}}} =\cap_{T\in \widetilde{\mathcal{S}}} \Fix T  = \cap^{t}_{i=1} \Fix F_{i} $.
	In addition, by  \cref{lemma:Propo:T:SUM:Produ}\cref{lemma:Propo:T:SUM:FixT:Produ},  $ \Fix A  = \cap^{t}_{i=1} \Fix F_{i}$. Hence, 	\cref{theo:CCS:LineaConve:F1t:Fix} is true.
	
	\cref{theo:CCS:LineaConve:F1t:LineaConv}: This follows from \cref{theo:CCS:LineaConve:F1t:Fix} above, \cref{lemma:Propo:T:SUM:Produ}\cref{lemma:Propo:T:SUM:T:Averaged:Produ}$\&$\cref{lemma:Propo:T:SUM:IncluAffi:Produ}, \cref{prop:TFNorm}\cref{prop:TFNorm:TTperp},    and \cref{fact:CCSLinearConvTSFirmNone}.
\end{proof}	

\begin{corollary} \label{cor:n:ccs:lineaconver}
	Suppose that $\mathcal{H} = \mathbb{R}^{n}$ and that $T_{1},T_{2}, \ldots, T_{m}$ are linear  isometries. Set $\mathcal{S}_{1} :=\{ \Id, T_{1}, T_{2}, \ldots, T_{m} \}$ and $\mathcal{S}_{2} :=\{ \Id, T_{1}, T_{2}T_{1}, \ldots, T_{m}\cdots T_{2}T_{1} \}$. Then 
\begin{enumerate}
	\item \label{cor:n:ccs:lineaconver:1} $(\CC{\mathcal{S}_{1}}^{k}x)_{k \in \mathbb{N}}$ converges linearly to $\Pro_{\cap^{m}_{i=1} \Fix T_{i}}x =\Pro_{ \Fix \CC{\mathcal{S}_{1}} }x $.
	\item  \label{cor:n:ccs:lineaconver:2} $(\CC{\mathcal{S}_{2}}^{k}x)_{k \in \mathbb{N}}$ converges linearly to $\Pro_{\cap^{m}_{i=1} \Fix T_{i}}x =\Pro_{ \Fix \CC{\mathcal{S}_{2}} }x$.
\end{enumerate}	
	
\end{corollary}
\begin{proof}
\cref{cor:n:ccs:lineaconver:1}: This is from \cref{theo:CCS:LineaConve:F1Fn}  with $\widetilde{\mathcal{S}}=\{ \Id, T_{1}, T_{2},\ldots, T_{m}\}$ by applying $t=m+1$, and  $F_{1}=\Id, F_{2}=T_{1}, \ldots, F_{t}=T_{m}$.

\cref{cor:n:ccs:lineaconver:2}: This  comes from 	\cref{theo:CCS:LineaConve:F1t}  with $\widetilde{\mathcal{S}}=\{ \Id, T_{1}, T_{2}T_{1}, \ldots, T_{m}\cdots T_{2}T_{1}\}$  by applying $t=m+1$, and  $F_{1}=\Id, F_{2}=T_{1}, F_{3}=T_{2}, \ldots, F_{t}=T_{m}$.
\end{proof}		
\begin{remark}
	\begin{enumerate}
		\item \cref{cor:n:ccs:lineaconver}\cref{cor:n:ccs:lineaconver:1} states  that for every nonempty set $\mathcal{S}$ of linear isometries in $\mathbb{R}^{n}$, if $\Id \in \mathcal{S}$, then for every $x \in \mathbb{R}^{n}$, $(\CC{\mathcal{S}}^{k}x)_{k \in \mathbb{N}}$ converges linearly to $\Pro_{\cap_{T \in \mathcal{S}} \Fix T }x =\Pro_{ \Fix \CC{\mathcal{S}} }x$.
		\item \cref{cor:n:ccs:lineaconver}\cref{cor:n:ccs:lineaconver:1}$\&$\cref{cor:n:ccs:lineaconver:2} illustrate that given arbitrary linear isometries  $T_{1},T_{2}, \ldots, T_{m}$  in $\mathbb{R}^{n}$,  we are able to construct multiple CIMs linearly converging to  $\Pro_{\cap^{m}_{i=1} \Fix T_{i}}x  =\Pro_{ \Fix \CC{\mathcal{S}} }x$ for every $x \in \mathbb{R}^{n}$.
	\end{enumerate}
\end{remark}	
\begin{example}
	Let $U_{1}$ and $U_{2}$ be closed linear subspaces of $\mathbb{R}^{n}$. Set $\mathcal{S}_{1} :=\{\Id, \R_{U_{1}}, \R_{U_{2}} \}$ and $\mathcal{S}_{2} := \{\Id, \R_{U_{1}}, \R_{U_{2}}\R_{U_{1}} \}$. Let $x \in \mathcal{H}$. Then by \cref{cor:n:ccs:lineaconver}, $(\CC{\mathcal{S}_{1}}^{k}x)_{k \in \mathbb{N}}$ and $(\CC{\mathcal{S}_{2}}^{k}x)_{k \in \mathbb{N}}$ both linearly converge to $\Pro_{U_{1} \cap U_{2}}x$.
\end{example}

\subsection*{Linear convergence of CIMs in Hilbert spaces with adjustment of the initial point} \label{sec:LCCIMH}

In view of \cite[Page~3438]{BDHP2003}, in order to better accelerate the symmetric MAP, the accelerated symmetric MAP first applies another operator to the initial point. (Similarly, to accelerate the DRM, the C--DRM  first applies another operator to the initial point as well, see \cite[Theorem~1]{BCS2017}.) 
The following results provide sufficient conditions for the linear convergence of CIMs with first applying an operator $T$ to the initial point. We shall provide applications of the following results later.

\begin{theorem} \label{theo:LineaConvIneq:TF}
	Suppose that $T_{1}, \ldots, T_{m}$ are linear isometries from $\mathcal{H} $ to $\mathcal{H}$ and that $ \mathcal{S} = \{ T_{1}, T_{2},  \ldots, T_{m} \} $ with $T_{1} = \Id$. Let $W$ be a nonempty closed linear subspace of $\cap^{m}_{i=1} \Fix T_{i}$. Let $F: \mathcal{H} \to \mathcal{H} $ satisfy $(\forall x \in \mathcal{H})$ $Fx \in \aff (\mathcal{S}(x))$. Let $T \in \mathcal{B} (\mathcal{H})$ be such that 
	$\Pro_{W}T=\Pro_{W}=T \Pro_{W}$.
	Assume one of the following items 
	holds:
	\begin{enumerate}
		\item \label{theo:LineaConvIneq:TF:W} There exists $\gamma \in \left[0,1\right[$ such that $(\forall x \in \mathcal{H})$  $\norm{Fx -\Pro_{W}x} \leq \gamma \norm{x -\Pro_{W}x}$. 
		\item \label{theo:LineaConvIneq:TF:FFixCap} There exists $\gamma \in \left[0,1\right[$ such that $(\forall x \in \mathcal{H})$  $\norm{Fx -\Pro_{\Fix F}x} \leq \gamma \norm{x -\Pro_{\Fix F}x}$ and that   $(\forall k \in \mathbb{N})$ $\Pro_{\Fix F} \CC{\mathcal{S}}^{k}T = \Pro_{W}$. 
	\end{enumerate}	
	Then 
	\begin{align} \label{eq:prop:LineaConvIneq:TF}
	(\forall x \in \mathcal{H})  (\forall k \in \mathbb{N})\quad \norm{ \CC{\mathcal{S}}^{k}Tx -\Pro_{W}x } \leq \gamma^{k} \norm{T\Pro_{W^{\perp}}} \norm{x - \Pro_{W}x }.
	\end{align}
\end{theorem}

\begin{proof}
	We prove \cref{eq:prop:LineaConvIneq:TF}  by induction on $k$. 
	
	Because
	\begin{align*}
	(\forall x \in \mathcal{H}) \quad 	\norm{Tx -\Pro_{W}x} &=  \norm{Tx -T\Pro_{W}x} \quad (\text{by } \Pro_{W} =T\Pro_{W} )\\
	& =  \norm{T(x -\Pro_{W}x)}  \quad (\text{$T$ is linear})\\
	&=  \norm{T\Pro_{W^{\perp}}\Pro_{W^{\perp}}x}   \quad (\text{by \cref{MetrProSubs8}\cref{MetrProSubs8:ii} and \cref{fact:ProjectorInnerPRod}\cref{fact:ProjectorInnerPRod:Idempotent}})\\
	&\leq  \norm{T\Pro_{W^{\perp}}} \norm{x - \Pro_{W}x},
	\end{align*} 
\cref{eq:prop:LineaConvIneq:TF} is true for $k=0$.
	
	Suppose that \cref{eq:prop:LineaConvIneq:TF} is true for some $k \in \mathbb{N}$. Let $x \in \mathcal{H}$. First note that
	\begin{align*}
	\norm{ \CC{\mathcal{S}}^{k+1}Tx -\Pro_{W}x } & = \Norm{ \CC{\mathcal{S}} \left( \CC{\mathcal{S}}^{k}Tx \right)  - \Pro_{W}(\CC{\mathcal{S}}^{k}Tx)}  \quad (\text{by  \cref{fact:CCS:proper:NormPres:T}\cref{fact:CCS:proper:NormPres:T:PaffU} and $\Pro_{W}T=\Pro_{W}$}) \\
	&\leq   \Norm{ F\left( \CC{\mathcal{S}}^{k}Tx \right)  - \Pro_{W}(\CC{\mathcal{S}}^{k}Tx)}.  \quad \left(\text{apply \cref{fact:CCS:proper:NormPres:T}\cref{fact:CCS:proper:NormPres:T:F} with $z=\Pro_{W}(\CC{\mathcal{S}}^{k}Tx)$}\right) 
	\end{align*}
	
	Assume first that assumption \cref{theo:LineaConvIneq:TF:W}  holds. Then
	\begin{align*}
	\norm{ F\left( \CC{\mathcal{S}}^{k}Tx \right)  - \Pro_{W}(\CC{\mathcal{S}}^{k}Tx)}  
	& \leq \gamma \Norm{ \CC{\mathcal{S}}^{k}Tx  - \Pro_{W}(\CC{\mathcal{S}}^{k}Tx)} \quad (\text{by assumption  \cref{theo:LineaConvIneq:TF:W}}) \\
	& \leq \gamma  \Norm{  \CC{\mathcal{S}}^{k}Tx  - \Pro_{W}x} \quad (\text{by  \cref{fact:CCS:proper:NormPres:T}\cref{fact:CCS:proper:NormPres:T:PaffU} and $\Pro_{W}T=\Pro_{W}$}) \\
	& \leq \gamma^{k+1} \norm{T\Pro_{W^{\perp}}} \norm{x - \Pro_{W}x }. \quad (\text{by inductive hypothesis})
	\end{align*}	
	
	Now assume that assumption \cref{theo:LineaConvIneq:TF:FFixCap} holds. By assumptions and by \cref{fact:CCS:proper:NormPres:T}\cref{fact:CCS:proper:NormPres:T:PaffU}, we have that
	\begin{align} \label{eq:prop:LineaConvIneq:TF:PWPCAPFixT}
	(\forall t \in \mathbb{N}) \quad \Pro_{\Fix F} \CC{\mathcal{S}}^{t}T = \Pro_{W} = \Pro_{W}T =
	\Pro_{W}\CC{\mathcal{S}}^{t} T.
	\end{align}
	Hence,
	\begin{align*}
	\norm{ F\left( \CC{\mathcal{S}}^{k}Tx \right)  - \Pro_{W}(\CC{\mathcal{S}}^{k}Tx)}  &= 
	\norm{ F\big( \CC{\mathcal{S}}^{k}Tx \big)  - \Pro_{\Fix F}(\CC{\mathcal{S}}^{k}Tx)}  \quad (\text{by  \cref{eq:prop:LineaConvIneq:TF:PWPCAPFixT}})\\
	& \leq \gamma \Norm{ \CC{\mathcal{S}}^{k}Tx   - \Pro_{\Fix F}(\CC{\mathcal{S}}^{k}Tx)} \quad (\text{by assumption  \cref{theo:LineaConvIneq:TF:FFixCap}}) \\
	&=  \gamma \Norm{ \CC{\mathcal{S}}^{k}Tx  - \Pro_{W}x} \quad (\text{by  \cref{eq:prop:LineaConvIneq:TF:PWPCAPFixT}})\\
	&\leq \gamma^{k+1} \norm{T\Pro_{W^{\perp}}} \norm{x - \Pro_{W}x }. \quad (\text{by  inductive hypothesis})
	\end{align*}	
	Altogether, the proof is complete.
\end{proof}

\begin{remark}
	\begin{enumerate}
		\item One application of \cref{theo:LineaConvIneq:TF}\cref{theo:LineaConvIneq:TF:W} is shown in \cref{theo:CCS:Accel:AT:Tx} below.
		\item	Let $L_{1}$ and $L_{2}$ be closed linear subspace in $\mathcal{H}$. Assume that $\widetilde{\mathcal{S}}$ is a finite subset of $\Omega( \R_{L_{1}}, \R_{L_{2}} )$, where $\Omega( \R_{L_{1}}, \R_{L_{2}} )$ consists of all finite compositions of operators from $\{\R_{L_{1}}, \R_{L_{2}} \}$.  Assume that $L_{1} \cap L_{2} \subseteq \cap_{T \in \widetilde{\mathcal{S}} } \Fix T$. Let $K$ be a closed linear subspace of $\mathcal{H}$ such that $L_{1} \cap L_{2} \subseteq  K \subseteq  L_{1} + L_{2} $. Denote by $T_{L_{2},L_{1}}$  the Douglas-Rachford operator associated with $L_{1}$ and $L_{2}$. Assume $T^{d}_{L_{2},L_{1}} \in \aff \widetilde{\mathcal{S}}$ for some $d \in \mathbb{N} \smallsetminus \{0\}$. By \cite[Corollary~5.17]{BOyW2019Isometry}, we know that $\Pro_{L_{1} \cap L_{2}} =\Pro_{\Fix T_{L_{2},L_{1}}}\Pro_{K} = \Pro_{L_{1} \cap L_{2}} \Pro_{K} = \Pro_{L_{1} \cap L_{2}} \CC{\widetilde{\mathcal{S}}}^{k}\Pro_{K} =\Pro_{\Fix T_{L_{2},L_{1}}} \CC{\widetilde{\mathcal{S}} }^{k}\Pro_{K}  $.
		In fact, \cite[Proposition~5.18]{BOyW2019Isometry} is a special case of  \cref{theo:LineaConvIneq:TF}\cref{theo:LineaConvIneq:TF:FFixCap} when 	$W= L_{1} \cap L_{2}$, $F= T^{d}_{L_{2},L_{1}} $ and $T =\Pro_{K}$.  
		Because
		 \cite[Proposition~5.18]{BOyW2019Isometry} is a generalization of \cite[Theorem~1]{BCS2017},  \cite[Theorem~1]{BCS2017} is also a special instance of  \cref{theo:LineaConvIneq:TF}\cref{theo:LineaConvIneq:TF:FFixCap}.
	\end{enumerate}
\end{remark}

\section{Linear convergence of  CRMs in Hilbert spaces}
\label{sec:LinearConverCRMHilert}
Since reflectors associated with affine subspaces are isometries, we deduce from \cref{fact:CCS:proper:NormPres:T}\cref{fact:CCS:proper:NormPres:T:prop} that all of the circumcenter mappings induced by finite sets of reflectors are proper. In particular, we call the circumcenter  method induced by a finite set of reflectors the \emph{circumcentered reflection method} (CRM). 

 In this section, we shall use the linear convergence of \emph{method of alternating projections} (MAP) to deduce sufficient conditions for the linear convergence of CRMs for finding the best approximation onto the intersection of finitely many affine subspaces.

\cref{prop:IsometryAffine}, \cref{lemma:TAffineFLinear}\cref{lemma:TAffineFLinear:TF} and \cref{theo:PARALLIsometryLinear} imply that in order to study the linear convergence of CRMs, we are free to assume that all of the related reflectors are associated with linear subspaces. 

Recall that $m \in \mathbb{N} \smallsetminus \{0\}$.
In this section, we assume that 
\begin{empheq}{equation*}
U_{1}, \ldots, U_{m}~\text{are closed linear subspaces in the real Hilbert space}~\mathcal{H}.
\end{empheq}
Clearly, $\{0\} \subseteq \cap^{m}_{i=1} U_{i} \neq \varnothing$. Set
\begin{empheq}[box=\mybluebox]{equation*} 
\Omega := \Omega(\R_{U_{1}}, \ldots, \R_{U_{m}} ) = \Big\{ \R_{U_{i_{r}}}\cdots \R_{U_{i_{2}}}\R_{U_{i_{1}}}  ~\Big|~ r \in \mathbb{N}, ~\mbox{and}~ i_{1}, \ldots,  i_{r} \in \{1, \ldots,m\}    \Big\},
\end{empheq}
and 
\begin{empheq}[box=\mybluebox]{equation*}
\Psi :=\Big\{ \R_{U_{i_{r}}}\cdots \R_{U_{i_{2}}}\R_{U_{i_{1}}}  ~\Big|~ r, i_{1}, i_{2}, \ldots, i_{r} \in \{0,1, \ldots,m\} ~\mbox{and}~0<i_{1}<\cdots<i_{r} \Big\}.
\end{empheq}
Recall that we use the empty product convention, so for $r=0$, $\R_{U_{i_{0}}}\cdots \R_{U_{i_{2}}}\R_{U_{i_{1}}} = \Id$.

We also assume that
\begin{empheq}[box=\mybluebox]{equation} \label{eq:Defi:S}
\Psi \subseteq \mathcal{S} \subseteq \Omega \quad \text{and} \quad \mathcal{S} ~\text{consists finitely many elements.}
\end{empheq}

Recall that $x \in \mathcal{H}$, 
\begin{empheq}{equation*}
\mathcal{S}(x) := \{Tx ~|~ T \in \mathcal{S}\}.
\end{empheq}

In this section, we will deduce some linear convergence results on CRMs induced by $\mathcal{S}$ satisfying \cref{eq:Defi:S}. We shall show that some CRMs do not have worse convergence rate than the sharp convergence rate of MAP for finding best approximation on $\cap^{m}_{i=1}U_{i}$. Moreover, we shall prove that some CRMs attain the known convergence rate of the accelerated symmetric MAP shown in \cite{BDHP2003}.

\begin{remark} \label{rem:Psi2Pm}
	We claim that there are exactly $2^m$ possible combinations for the indices of the reflectors making up the elements of the set $\Psi$. 
	
	In fact, for every $r \in \I := \{0,1, \ldots, m\}$, the $r$-combination of the set $\I$ is a subset of $r$ distinct items of $\I$.\footnotemark ~In addition, the number of $r$-combinations of $\I$ equals to the binomial coefficient $\binom{m}{r}$. 
	\footnotetext{Recall that we use the empty product convention that $\prod^{0}_{j=1}\R_{U_{i_{j}}} =\Id$, so the $0$-combination of the set $\I$ is the $\Id$.}
	Moreover, by the Binomial Theorem, 
	\begin{align*}
	2^{m} = (1+1)^{m}=\sum^{m}_{r=0} \binom{m}{r}.
	\end{align*} 
	Therefore, the claim is true. 
	
	Actually with consideration of duplication, there are at most $2^m$ pairwise distinct elements in $\Psi$. (For instance, if $U_{1}=U_{2}$, then $\R_{U_{2}}\R_{U_{1}} =\Id$.)
\end{remark}

For example, when $m=1$, $\Psi= \{\Id, \R_{U_{1}}\}$. When $m=2$, 
\begin{align*}
\Psi = \{\Id, \R_{U_{1}}, \R_{U_{2}}, \R_{U_{2}}\R_{U_{1}}\}.
\end{align*}
When $m=3$,
\begin{align*}
\Psi = \{\Id,\R_{U_{1}}, \R_{U_{2}}, \R_{U_{3}}, \R_{U_{2}}\R_{U_{1}}, \R_{U_{3}}\R_{U_{1}}, \R_{U_{3}}\R_{U_{2}}, \R_{U_{3}}\R_{U_{2}}\R_{U_{1}}\}.
\end{align*}

\subsection*{Examples of linear convergent CRMs}
First, let's see two examples where $m=2$ to get some intuition about our upcoming main result \cref{thm:MAP:LC}. Actually, these examples are also corollaries of \cref{thm:MAP:LC} below. 

\begin{example} \label{exam: lc:cw1} Assume that $m=2$, that $\mathcal{S} := \{\Id, \R_{U_{1}}, \R_{U_{2}}, \R_{U_{2}}\R_{U_{1}} \}$, and that $U_{1}+ U_{2}$ is closed. Set $\gamma := \norm{ \Pro_{U_{2}} \Pro_{U_{1}} \Pro_{(U_{1} \cap U_{2})^{\perp}} } $. Then $\gamma \in \left[0,1\right[$ and
	\begin{align*} 
	(\forall x \in \mathcal{H})  (\forall k \in \mathbb{N}) \quad \norm{\CC{\mathcal{S}}^{k}x-\Pro_{U_{1} \cap U_{2}} x} \leq \gamma^{k} \norm{x -\Pro_{U_{1} \cap U_{2}} x}.
	\end{align*}
	Consequently, $(\CC{\mathcal{S}}^{k}x)_{k \in \mathbb{N}}$ converges to $\Pro_{U_{1}\cap U_{2}}x$ with a linear rate $\gamma$.
\end{example}

\begin{proof}
	By assumption and \Cref{fac:PVPUPUcapVperp,fac:cFLess1}, we know $\gamma \in \left[0,1\right[\,$. Set
	$T_{\mathcal{S}} := \Pro_{U_{2}}\Pro_{U_{1}}$.
	Using \cref{fac:PUmU1FixT}, we get that $\Fix T_{\mathcal{S}} = \Fix \Pro_{U_{2}} \Pro_{U_{1}}  = U_{1} \cap U_{2} $.  Hence, apply  \cref{fac:LineNonexOperNorm}\cref{fac:LineNonexOperNorm:i} with $T$ replaced by $T_{\mathcal{S}}$ to obtain that
	\begin{align} \label{exam: lc:cw1:TS}
	(\forall x \in \mathcal{H})  (\forall k \in \mathbb{N}) \quad \norm{T_{\mathcal{S}}^{k}x-\Pro_{U_{1} \cap U_{2}}x} \leq \gamma^{k} \norm{x -\Pro_{U_{1} \cap U_{2}}x}.
	\end{align}
	Moreover, because
	\begin{align*}
	T_{\mathcal{S}} & = \Pro_{U_{2}}\Pro_{U_{1}} 
	= \frac{1}{2} (\R_{U_{2}} +\Id ) \frac{1}{2} (\R_{U_{1}}  +\Id ) = \frac{1}{2^{2}}(\R_{U_{2}}\R_{U_{1}}+\R_{U_{2}} +\R_{U_{1}}+ \Id) \in \aff (\mathcal{S})
	\end{align*}
	and  $U_{1}\cap U_{2}$ is a closed linear subspace of $\cap_{T \in \mathcal{S}} \Fix T$, the desired results are from \cref{fact:CWP:line:conv} and \cref{exam: lc:cw1:TS}.
\end{proof}

\begin{example} \label{exam: lc:cw2} 
	Assume that $m=2$, that $\mathcal{S} :=\{\Id, \R_{U_{1}}, \R_{U_{2}},\R_{U_{1}}\R_{U_{2}}, \R_{U_{2}}\R_{U_{1}}, \R_{U_{1}}\R_{U_{2}}\R_{U_{1}} \}$, and that $U_{1}+ U_{2}$ is closed. Set
$
	\gamma := \norm{ \Pro_{U_{2}} \Pro_{U_{1}} \Pro_{(U_{1} \cap U_{2})^{\perp}} }.
	$
	Then $\gamma \in\left[0,1\right[$ and
	\begin{align*} 
	(\forall x \in \mathcal{H})  (\forall k \in \mathbb{N}) \quad \norm{\CC{\mathcal{S}}^{k}x-\Pro_{U_{1} \cap U_{2}} x} \leq \gamma^{2k} \norm{x -\Pro_{U_{1} \cap U_{2} } x}.
	\end{align*}
	Consequently, $(\CC{\mathcal{S}}^{k}x)_{k \in \mathbb{N}}$ converges to $\Pro_{U_{1}\cap U_{2}} x$ with a linear rate $\gamma^{2}$.
\end{example}

\begin{proof}
	Denote $T_{\mathcal{S}} := \Pro_{U_{1}}\Pro_{U_{2}}\Pro_{U_{1}}$. Then $(\Pro_{U_{2}}\Pro_{U_{1}})^{*}\Pro_{U_{2}}\Pro_{U_{1}}=\Pro_{U_{1}}\Pro_{U_{2}}\Pro_{U_{1}} = T_{\mathcal{S}}$. Similarly with the proof of \cref{exam: lc:cw1}, we have that $\gamma \in \left[0,1\right[$ and  $\Fix T_{\mathcal{S}} = U_{1} \cap U_{2} $. Apply \cref{fac:LineNonexOperNorm}\cref{cor:LinNoneInequaTstarT} with $T$ replaced by $\Pro_{U_{2}}\Pro_{U_{1}}$ to obtain
	\begin{align} \label{eq:exam: lc:cw2:TS} 
	(\forall x \in \mathcal{H})  (\forall k \in \mathbb{N}) \quad \norm{T^{k}_{\mathcal{S}}x- \Pro_{U_{1}\cap U_{2}} x } 
	\leq & \gamma^{2k} \norm{x -\Pro_{U_{1}\cap U_{2}}x}.
	\end{align}
	Because
	\begin{align*}
	& \frac{1}{2^{3}}(\R_{U_{1}}+\Id) (\R_{U_{2}}+\Id)(\R_{U_{1}}+\Id)\\
	=~ & \frac{1}{2^{3}}(\R_{U_{1}}\R_{U_{2}}\R_{U_{1}}+\R_{U_{1}}\R_{U_{2}} +\R_{U_{1}}\R_{U_{1}} +\R_{U_{1}}+\R_{U_{2}}\R_{U_{1}}+ \R_{U_{2}}+\R_{U_{1}}+\Id) \\
	= ~&   \frac{1}{2^{3}}(\R_{U_{1}}\R_{U_{2}}\R_{U_{1}}+\R_{U_{1}}\R_{U_{2}} +\R_{U_{2}}\R_{U_{1}}+ 2\R_{U_{1}}+\R_{U_{2}}+2\Id) \\
	\in ~& \aff \{\Id, \R_{U_{1}}, \R_{U_{2}},\R_{U_{1}}\R_{U_{2}}, \R_{U_{2}}\R_{U_{1}}, \R_{U_{1}}\R_{U_{2}}\R_{U_{1}} \} =\aff \mathcal{S},
	\end{align*}
	we obtain
	\begin{align*}
	T_{\mathcal{S}} = \Pro_{U_{1}}\Pro_{U_{2}}\Pro_{U_{1}} =\frac{1}{2^{3}}(\R_{U_{1}}+\Id) (\R_{U_{2}}+\Id)(\R_{U_{1}}+\Id) \in \aff \mathcal{S}.
	\end{align*}
	Because $U_{1}\cap U_{2}$ is a closed linear subspace of $\cap_{T \in \mathcal{S}} \Fix T$, the results come from  \cref{fact:CWP:line:conv} and \cref{eq:exam: lc:cw2:TS}.
\end{proof}
\begin{remark}
	\begin{enumerate}
		\item From \cite[Theorem~9.31]{D2012} and \cref{fac:PVPUPUcapVperp}, we know that the sharp convergence rate of MAP associated with the two linear subspaces $U_{1}$ and $U_{2}$ 
		is $\norm{ \Pro_{U_{2}} \Pro_{U_{1}} \Pro_{ (U_{1} \cap U_{2} )^{\perp}} }^{2}$.
		 \cref{exam: lc:cw2}  tells us that the linear convergence rate of some CRMs is no worse than $\norm{ \Pro_{U_{2}} \Pro_{U_{1}} \Pro_{ (U_{1} \cap U_{2} )^{\perp}} }^{2}$.
		\item Set $\mathcal{S}_{1} := \{\Id, \R_{U_{1}}, \R_{U_{2}}, \R_{U_{2}}\R_{U_{1}} \}$ and $\mathcal{S}_{2} :=\{\Id, \R_{U_{1}}, \R_{U_{2}},\R_{U_{1}}\R_{U_{2}}, \R_{U_{2}}\R_{U_{1}}, \R_{U_{1}}\R_{U_{2}}\R_{U_{1}} \}$. In \cite[Section~6]{BOyW2019Isometry},  our \emph{numerical} experiments in $\mathbb{R}^{1000}$ showed that  the CRMs induced by  $\mathcal{S}_{1} $ and $\mathcal{S}_{2} $ given above perform better than the   DRM, MAP and the C-DRM introduced in \cite{BCS2017}.  In \cite{BOyW2019Isometry}, we didn't provide any \emph{analytical} expanation for the outstanding performances of the CRMs induced by $\mathcal{S}_{1}$ and $\mathcal{S}_{2}$. Now \Cref{exam: lc:cw1,exam: lc:cw2} present the theoretical support for the impressive performance.
	\end{enumerate}	
\end{remark}

\subsection*{CRMs associated with finitely many linear subspaces}
In order to prove our more general results, we need the following lemma, which is also interesting itself.
\begin{lemma} \label{lem:PR:induc}
	Recall that $m \in \mathbb{N} \smallsetminus \{0\}$, $U_{1}, \ldots, U_{m}$ are closed linear subspaces in $\mathcal{H}$ and $\Psi \subseteq \mathcal{S} \subseteq \Omega$. Then the following statements hold:
	\begin{enumerate}
		\item \label{lem:PR:induc:R}
		\begin{align} \label{eq:lem:PR:induc:R}
		\R_{U_{m}} \R_{U_{m-1}}\cdots \R_{U_{1}}
		= \scalemath{0.9}{2^{m} \Pro_{U_{m}} \Pro_{U_{m-1}}\cdots \Pro_{U_{1}} - \Big(  \Id + \sum^{m-1}_{k=1} \sum _{ \substack{i_{1}, \ldots, i_{k-1} ,i_{k} \in  \{1,2,\ldots,m\} \\ i_{1}<\cdots<i_{k-1} <i_{k} } } \R_{U_{i_{k}}} \R_{U_{i_{k-1}}} \cdots \R_{U_{i_{1}}} \Big)}.
		\end{align}
		\item \label{lem:PR:induc:P}
		\begin{align}   \label{eq:combination:aff}
		\Pro_{U_{m}} \Pro_{U_{m-1}}\cdots \Pro_{U_{1}} 
		= \frac{1}{2^{m}} \Big( \sum^{m}_{k=0} \sum _{ \substack{i_{1}, \ldots, i_{k-1} ,i_{k} \in \{1,2,\ldots,m\} \\ i_{1}<\cdots<i_{k-1} <i_{k} } } \R_{U_{i_{k}}} R_{U_{i_{k-1}}} \cdots \R_{U_{i_{1}}} \Big).
		\end{align}
		\item \label{lem:PR:induc:Inclu} $\Pro_{U_{m}} \Pro_{U_{m-1}}\cdots \Pro_{U_{1}}  \in \conv \Psi \subseteq \aff \mathcal{S}.$
	\end{enumerate}
\end{lemma}

\begin{proof}
	When $k=m$, the only possibility for   $\R_{U_{i_{k}}} \R_{U_{i_{k-1}}} \cdots \R_{U_{i_{1}}}$ with $i_{1}, \ldots, i_{k} \in \{1,2,\ldots,m\} $ and $i_{1}<\cdots<i_{k}$ is $\R_{U_{m}} \R_{U_{m-1}}\cdots \R_{U_{1}}$. Hence, clearly \cref{lem:PR:induc:R} $\Leftrightarrow$ \cref{lem:PR:induc:P}.  We thus only prove \cref{lem:PR:induc:R} and \cref{lem:PR:induc:Inclu}.

	\cref{lem:PR:induc:R}: We prove this by induction on $m$. If $m=1$, then by definition, $\R_{U_{1}}=2 \Pro_{U_{1}}- \Id$, which means that \cref{eq:lem:PR:induc:R} is true for $m=1$. 
	Now assume \cref{eq:lem:PR:induc:R} is true for some $m \geq 1$, i.e., 
	\begin{align}\label{eq:lem:PR:induc}
	\R_{U_{m}} \R_{U_{m-1}}\cdots \R_{U_{1}}
	= 2^{m} \Pro_{U_{m}} \Pro_{U_{m-1}}\cdots \Pro_{U_{1}} - \Big( \Id +  \sum^{m-1}_{k=1} \sum _{ \substack{i_{1}, \ldots, i_{k} \in \{1,2,\ldots,m\} \\ i_{1}<\cdots<i_{k-1} <i_{k} } } \R_{U_{i_{k}}} \R_{U_{i_{k-1}}} \cdots \R_{U_{i_{1}}} \Big)
	\end{align}
Then
	\begin{align*}
	& \quad \quad \R_{U_{m+1}}\R_{U_{m}} \R_{U_{m-1}}\cdots \R_{U_{1}} \\
	&\stackrel{\text{\cref{eq:lem:PR:induc}}}{=}  \R_{U_{m+1}} \Bigg(  2^{m} \Pro_{U_{m}}\Pro_{U_{m-1}}\cdots \Pro_{U_{1}} - \Big(\Id + \sum^{m-1}_{k=1} \sum _{ \substack{i_{1}, \ldots, i_{k} \in \{1,2,\ldots,m\} \\ i_{1}<\cdots<i_{k-1} <i_{k} } } \R_{U_{i_{k}}} \R_{U_{i_{k-1}}} \cdots \R_{U_{i_{1}}} \Big) \Bigg)\\
	&= 2^{m}   \R_{U_{m+1}}\Pro_{U_{m}} \Pro_{U_{m-1}}\cdots \Pro_{U_{1}} 
	-  \Big( \R_{U_{m+1}} +  \sum^{m-1}_{k=1} \sum _{ \substack{i_{1}, \ldots, i_{k} \in \{1,2,\ldots,m\} \\ i_{1}<\cdots<i_{k-1} <i_{k} } } \R_{U_{m+1}}\R_{U_{i_{k}}} \R_{U_{i_{k-1}}} \cdots \R_{U_{i_{1}}} \Big)\\
	&= 2^{m} (2\Pro_{U_{m+1}}-\Id ) \Pro_{U_{m}} \Pro_{U_{m-1}}\cdots \Pro_{U_{1}} 
	-  \Big( \R_{U_{m+1}} + \sum^{m-1}_{k=1} \sum _{ \substack{i_{1}, \ldots, i_{k} \in \{1,2,\ldots,m\} \\ i_{1}<\cdots<i_{k-1} <i_{k} } } \R_{U_{m+1}}\R_{U_{i_{k}}} \R_{U_{i_{k-1}}} \cdots \R_{U_{i_{1}}} \Big)\\
	&= 2^{m+1} \Pro_{U_{m+1}} \Pro_{U_{m}} \cdots \Pro_{U_{1}} - 2^{m}  \Pro_{U_{m}} \Pro_{U_{m-1}}\cdots \Pro_{U_{1}} 
	-  \Big( \R_{U_{m+1}} +\sum^{m-1}_{k=1} \sum _{ \substack{i_{1}, \ldots, i_{k} \in \{1,2,\ldots,m\} \\ i_{1}<\cdots<i_{k-1} <i_{k} } } \R_{U_{m+1}}\R_{U_{i_{k}}} \R_{U_{i_{k-1}}} \cdots \R_{U_{i_{1}}} \Big)\\
	&\stackrel{\text{\cref{eq:lem:PR:induc}}}{=} 2^{m+1} \Pro_{U_{m+1}} \Pro_{U_{m}} \cdots \Pro_{U_{1}} 
	-  \R_{U_{m}} \R_{U_{m-1}}\cdots \R_{U_{1}} - \Big( \Id + \sum^{m-1}_{k=1} \sum _{ \substack{i_{1}, \ldots, i_{k} \in \{1,2,\ldots,m\} \\ i_{1}<\cdots<i_{k-1} <i_{k} } } \R_{U_{i_{k}}} \R_{U_{i_{k-1}}} \cdots \R_{U_{i_{1}}} \Big)\\
	&\quad -  \Big( R_{U_{m+1}} +\sum^{m-1}_{k=1} \sum _{ \substack{i_{1}, \ldots, i_{k} \in \{1,2,\ldots,m\} \\ i_{1}<\cdots<i_{k-1} <i_{k} } } \R_{U_{m+1}}\R_{U_{i_{k}}} \R_{U_{i_{k-1}}} \cdots \R_{U_{i_{1}}}  \Big)\\
	&= 2^{m+1} \Pro_{U_{m+1}} \Pro_{U_{m}}\cdots \Pro_{U_{1}} -\Big( \Id + \sum^{m}_{k=1} \sum _{ \substack{i_{1}, \ldots, i_{k} \in \{1,2,\ldots,m+1\} \\ i_{1}<\cdots<i_{k-1} <i_{k} } } \R_{U_{i_{k}}} \R_{U_{i_{k-1}}} \cdots \R_{U_{i_{1}}}  \Big),
	\end{align*}
	which is \cref{eq:lem:PR:induc:R} with $m$ being replaced by $m+1$. Therefore, \cref{lem:PR:induc:R} is true.

	\cref{lem:PR:induc:Inclu}: By \cref{rem:Psi2Pm}, we know there are exactly $2^{m}$ items in the big bracket on the right--hand side of \cref{eq:combination:aff}, since these items in the big bracket are exactly all of the items in the set $\Psi$. Hence,
	\begin{align*}
	\Pro_{U_{m}} \Pro_{U_{m-1}}\cdots \Pro_{U_{1}} 
	=  \frac{1}{2^{m}} \Big( \sum^{m}_{k=0} \sum _{ \substack{i_{1}, \ldots, i_{k-1} ,i_{k} \in \{1,2,\ldots,m\} \\ i_{1}<\cdots<i_{k-1} <i_{k} } } \R_{U_{i_{k}}} \R_{U_{i_{k-1}}} \cdots \R_{U_{i_{1}}} \Big)  \in \conv \Psi \subseteq \aff \mathcal{S}.
	\end{align*}
	Therefore, the proof is complete.
\end{proof} 

Now we are ready to use results on the linear convergence of MAPs or symmetric MAPs to prove the linear convergence of CRMs.
\begin{theorem} \label{thm:MAP:LC}
	Recall that $U_{1}, \ldots ,U_{m}$ are closed linear subspaces of $\mathcal{H}$ and that $\Psi \subseteq \mathcal{S} \subseteq \Omega$. Set   $\gamma := \norm{ \Pro_{U_{m}} \Pro_{U_{m-1}} \cdots \Pro_{U_{1}} \Pro_{(\cap^{m}_{i=1} U_{i} )^{\perp}} } $.  Assume that $m \geq 2$ and that $U^{\perp}_{1} + \cdots+ U^{\perp}_{m}$ is closed. Then $\gamma \in \left[0,1\right[$ and
	\begin{align*}
	(\forall x \in \mathcal{H} ) (\forall k \in \mathbb{N}) \quad \norm{\CC{\mathcal{S}}^{k}x- \Pro_{\cap^{m}_{i=1} U_{i}} x} \leq \gamma^{k} \norm{x - \Pro_{\cap^{m}_{i=1} U_{i}} x}.
	\end{align*}
	Consequently, $(\CC{\mathcal{S}}^{k}x)_{k \in \mathbb{N}}$ converges to $\Pro_{\cap^{m}_{i=1} U_{i}} x$ with a linear rate  $\gamma$.
\end{theorem}

\begin{proof}
	Set $T_{\mathcal{S}} :=\Pro_{U_{m}} \Pro_{U_{m-1}} \cdots \Pro_{U_{1}}$.  \cref{fac:PUmU1FixT} yields $\Fix T_{\mathcal{S}} = \cap^{m}_{i=1}U_{i} $.  The assumptions and \cref{cor:AnglN-tupleLess1:P} imply
	\begin{align} \label{thm:MAP:LC:LineRate}
	\gamma = \norm{T_{\mathcal{S}} \Pro_{(\Fix T_{\mathcal{S}} )^{\perp}}}  \in \left[0,1\right[\,.
	\end{align}
	Applying  \cref{fac:LineNonexOperNorm}\cref{fac:LineNonexOperNorm:i} with $T$ replaced by $T_{\mathcal{S}}$, we obtain
	\begin{align} \label{thm:MAP:LC:TS}
	(\forall x \in \mathcal{H} )  (\forall k \in \mathbb{N}) \quad \norm{T^{k}_{\mathcal{S}}x-\Pro_{\cap^{m}_{i=1} U_{i} }x} \leq \gamma^{k} \norm{x -\Pro_{\cap^{m}_{i=1} U_{i} }x}.
	\end{align}
	By \cref{lem:PR:induc}\cref{lem:PR:induc:Inclu}, $T_{\mathcal{S}} \in \aff ( \mathcal{S} )$. By the construction of $\Omega$ and by $\Psi \subseteq \mathcal{S} \subseteq \Omega$, we obtain that $(\forall T \in \mathcal{S})$,  $\cap^{m}_{i=1} U_{i} \subseteq \Fix T$, which implies that $\cap^{m}_{i=1} U_{i}$ is a closed linear subspace of $\cap_{T \in  \mathcal{S}} \Fix T$. Hence, \cref{fact:CWP:line:conv}, \cref{thm:MAP:LC:LineRate} and \cref{thm:MAP:LC:TS} yield the required results.
\end{proof}

\begin{corollary} \label{cor:SymMAP:LC}
	Assume that $m=2n-1$ for some $n \in \mathbb{N} \smallsetminus \{0\}$, that $U_{1}, \ldots ,U_{n}$ are closed linear subspaces of $\mathcal{H}$ with $U^{\perp}_{1} + \cdots + U^{\perp}_{n}$ being closed, and $(\forall i \in \{1, \ldots, n-1\})$ $U_{n+i} :=U_{n-i}$.  Recall that $\Psi \subseteq \mathcal{S} \subseteq \Omega$.
	Denote $\gamma := \norm{ \Pro_{U_{n}} \Pro_{U_{n-1}} \cdots \Pro_{U_{1}} \Pro_{(\cap^{n}_{i=1} U_{i})^{\perp}} } $. Then $\gamma \in \left[0,1\right[$ and
	\begin{align*}
	(\forall x \in \mathcal{H})  (\forall k \in \mathbb{N}) \quad \norm{\CC{\mathcal{S}}^{k}x- \Pro_{\cap^{n}_{i=1} U_{i}} x} \leq \gamma^{2k} \norm{x - \Pro_{\cap^{n}_{i=1} U_{i}} x},
	\end{align*}
	that is $(\CC{\mathcal{S}}^{k}x)_{k \in \mathbb{N}}$ converges to $\Pro_{\cap^{n}_{i=1} U_{i}} x$ with a linear rate  $\gamma^{2}$.
\end{corollary}

\begin{proof}
	First note that,
	$
	\Pro_{U_{m}} \Pro_{U_{m-1}} \cdots \Pro_{U_{1}} = \Pro_{U_{1}}\cdots \Pro_{U_{n-1}} \Pro_{U_{n}} \Pro_{U_{n-1}} \cdots \Pro_{U_{1}},
$
	and that
$
	U^{\perp}_{1} + \cdots + U^{\perp}_{n-1} + U^{\perp}_{n} +U^{\perp}_{n+1}+\cdots+ U^{\perp}_{m} = U^{\perp}_{1} + \cdots+ U^{\perp}_{n}$. 
	Set $ \rho: = \norm{ \Pro_{U_{m}} \Pro_{U_{m-1}} \cdots \Pro_{U_{1}} \Pro_{(\cap^{m}_{i=1} U_{i} )^{\perp}} }$. Since $U^{\perp}_{1} + \cdots+ U^{\perp}_{n}$ is closed, \cref{thm:MAP:LC} implies
	\begin{align} \label{eq:cor:SymMAP:LC:CCSk}
	(\forall x \in \mathcal{H} ) (\forall k \in \mathbb{N}) \quad \norm{\CC{\mathcal{S}}^{k}x- \Pro_{\cap^{m}_{i=1} U_{i}} x} \leq \rho^{k} \norm{x - \Pro_{\cap^{m}_{i=1} U_{i}} x}.
	\end{align}
Also set $T :=\Pro_{U_{n}} \Pro_{U_{n-1}} \cdots \Pro_{U_{1}}$. Then \cref{fac:PUmU1FixT}, \cref{fac:LineNonexOperNorm}\cref{fac:LineNonexOperNorm:iv} and \cref{cor:AnglN-tupleLess1:P} yield
	\begin{align} \label{eq:cor:SymMAP:LC:gamma}
	\gamma =\norm{T\Pro_{(\cap^{n}_{i=1} U_{i})^{\perp}} }=\norm{T^{*}T\Pro_{(\cap^{n}_{i=1} U_{i})^{\perp}} }^{\frac{1}{2}}   = \rho^{\frac{1}{2}}  \in \left[0,1\right[\, ,
	\end{align}
	because $\cap^{n}_{i=1} U_{i} =\cap^{m}_{i=1} U_{i}$. Hence, \cref{eq:cor:SymMAP:LC:CCSk} and \cref{eq:cor:SymMAP:LC:gamma}  yield
	\begin{align*}
	(\forall x \in \mathcal{H} )  (\forall k \in \mathbb{N}) \quad \norm{\CC{\mathcal{S}}^{k}x- \Pro_{\cap^{n}_{i=1} U_{i}} x} \leq \gamma^{2k} \norm{x - \Pro_{\cap^{n}_{i=1} U_{i}} x},
	\end{align*}
	as claimed.
\end{proof}

\subsection*{Applications of the accelerated symmetric MAP}

In this section, set $T:=\Pro_{U_{m}}\cdots \Pro_{U_{1}}$ and let $A_{T}$  be the accelerated mapping of $T$ defined in \cref{defn:AccelrAT}. In the following two results we take advantage of the linear convergence of iteration sequence from $A_{T}$ as a bridge to show the linear convergence of certain classes of  CRMs.

\begin{theorem} \label{theo:CCS:Accel:AT:x}
	Assume that $m=2n-1$ for some $n \in \mathbb{N} \smallsetminus \{0\}$, $U_{1}, \ldots ,U_{n}$ are closed linear subspaces of $\mathcal{H}$ with $U^{\perp}_{1} + \cdots + U^{\perp}_{n}$ being closed, and $(\forall i \in \{1, \ldots, n-1\})$ $U_{n+i} :=U_{n-i}$.  Recall that $\Psi \subseteq \mathcal{S} \subseteq \Omega$. Let $c_{1}$ and $c_{2}$ be defined as in 
	\cref{eq:fac:AT:c1:c2:c1} and \cref{eq:fac:AT:c1:c2:c2}. 
	Set $T:=\Pro_{U_{m}}\cdots \Pro_{U_{1}} =\Pro_{U_{1}} \cdots \Pro_{U_{n-1}} \Pro_{U_{n}} \Pro_{U_{n-1}} \cdots \Pro_{U_{1}}$, $c(T) :=\norm{T\Pro_{(\Fix T)^{\perp}}}$,  $\gamma := \norm{ \Pro_{U_{n}} \Pro_{U_{n-1}} \cdots \Pro_{U_{1}} \Pro_{(\cap^{m}_{i=1}U_{i})^{\perp}} } $, and $\eta :=\frac{c_{2}-c_{1}}{2-c_{1}-c_{2}}$.
	Then the following  statements hold:
	\begin{enumerate}
		\item \label{theo:CCSAccel:AT:x:C} $0 \leq \eta \leq \frac{c(T)}{2-c(T)} \leq c(T) =\gamma^{2}  <1$.
		\item \label{theo:CCSAccel:AT:x:Ineq}
		$(\forall x \in \mathcal{H}) (\forall k \in \mathbb{N}) \quad \norm{\CC{\mathcal{S}}^{k}(x) - \Pro_{\cap^{n}_{i=1}U_{i}}x } \leq \eta^{k} \norm{x - \Pro_{\cap^{n}_{i=1}U_{i}}x}.
		$
	\end{enumerate}
\end{theorem}

\begin{proof}
	By \Cref{fact:ProjectorInnerPRod,MetrProSubs8}, 
	we know that $T$ is a linear, nonexpansive and  self-adjoint. Because $T=\Pro_{U_{1}} \cdots \Pro_{U_{n-1}} \Pro_{U_{n}} \Pro_{U_{n-1}} \cdots \Pro_{U_{1}} =(\Pro_{U_{n}} \Pro_{U_{n-1}} \cdots \Pro_{U_{1}} )^{*}\Pro_{U_{n}} \Pro_{U_{n-1}} \cdots \Pro_{U_{1}} $, by \cite[Example~20.16(ii)]{BC2017}, $T$ is monotone.
	
	\cref{theo:CCSAccel:AT:x:C}: By \cref{fac:PUmU1FixT}, $ \Fix T = \cap^{n}_{i=1}U_{i} \neq \varnothing$. By \cref{fac:LineNonexOperNorm}\cref{fac:LineNonexOperNorm:iv} and \cref{cor:AnglN-tupleLess1:P}, we know that $c(T) =\gamma^{2}  <1$. Hence, 
	the inequalities follow from \cref{fac:AT:c1c2:cT}.
	
	\cref{theo:CCSAccel:AT:x:Ineq}:
	Let $A_{T}$  be the accelerated mapping defined in \cref{defn:AccelrAT} of $T$. For every $x \in \mathcal{H}$, since $x \in \aff \mathcal{S}(x)$, and since by  \cref{lem:PR:induc}, $Tx \in \aff \mathcal{S}(x)$, thus $A_{T}x \in \aff \{x, Tx\} \subseteq \aff \mathcal{S}(x)$.
	Since $ \Fix T = \cap^{m}_{i=1}U_{i} $, using  \cref{fac:AT:Norm:MPerp}, we obtain
	\begin{align*}  
	(\forall x \in \mathcal{H}) \quad \norm{A_{T} x - \Pro_{\cap^{m}_{i=1}U_{i}} x}   \leq \eta \norm{ x - \Pro_{\cap^{m}_{i=1}U_{i}} x}.
	\end{align*}
	As we proved in \cref{thm:MAP:LC}, the assumption  $\Psi \subseteq \mathcal{S} \subseteq \Omega$  implies that $\cap^{m}_{i=1} U_{i}$ is a closed linear subspace of $\cap_{G \in  \mathcal{S}} \Fix G$.
	Hence, apply \cref{fact:CWP:line:conv} with $F=A_{T} $ and $ W =\cap^{m}_{i=1}U_{i}$
to obtain \cref{theo:CCSAccel:AT:x:Ineq}.
\end{proof}

\begin{example} {\rm \cite[page 3438]{BDHP2003}} \label{exam:counterT}
	Let $T$ be the product of two orthogonal projections onto two $1$--dimensional (nonorthogonal) subspaces in the Euclidean plane. Then the accelerated algorithm, $(A_{T}^{k}(Tx))_{k \in \mathbb{N}}$, converges in two steps, that is, $A_{T}(Tx)=\Pro_{\Fix T}x$ for any starting point. However, for any choice of $x$ which is not in the range of $T$, none of the terms of the sequence $(A_{T}^{k}(x))_{k \in \mathbb{N}}$ is equal to  $\Pro_{\Fix T}x$, which means that $(A_{T}^{k}(x))_{k \in \mathbb{N}}$ does not converge to $\Pro_{\Fix T}x$ in a finite number of steps.
\end{example}

Inspired by \cref{exam:counterT}, \cref{fac:AT:Norm:MPerp}  and \cref{theo:LineaConvIneq:TF}\cref{theo:LineaConvIneq:TF:W}, we show the following result, where we consider the special initial point
 $x_{0} =\Pro_{U_{m}}\cdots \Pro_{U_{1}}x$. 
 
\begin{theorem} \label{theo:CCS:Accel:AT:Tx}
	Assume $m=2n-1$ for some $n \in \mathbb{N} \smallsetminus \{0\}$, $U_{1}, \ldots ,U_{n}$ are closed linear subspaces of $\mathcal{H}$ with $U^{\perp}_{1} + \cdots + U^{\perp}_{n}$ being closed, $(\forall i \in \{1, \ldots, n-1\})$ $U_{n+i} :=U_{n-i}$. Recall that $\Psi \subseteq \mathcal{S} \subseteq \Omega$.
	Let $c_{1}$ and $c_{2}$ be defined as in 
	\cref{eq:fac:AT:c1:c2:c1} and \cref{eq:fac:AT:c1:c2:c2}.  Set $T:=\Pro_{U_{m}}\cdots \Pro_{U_{1}}=\Pro_{U_{1}} \cdots \Pro_{U_{n-1}}\Pro_{U_{n}} \Pro_{U_{n-1}} \cdots \Pro_{U_{1}}$,  $c(T) :=\norm{T\Pro_{(\Fix T)^{\perp}}}$,  $\gamma := \norm{ \Pro_{U_{n}} \Pro_{U_{n-1}} \cdots \Pro_{U_{1}} \Pro_{(\cap^{m}_{i=1}U_{i})^{\perp}} } $, and  $\eta :=\frac{c_{2}-c_{1}}{2-c_{1}-c_{2}} $.
	Then $(\forall k \in \mathbb{N})$  $\eta^{k}c(T) \leq \gamma^{2(k+1)}$, and 
	\begin{align*} 
	(\forall x \in \mathcal{H})  (\forall k \in \mathbb{N}) \quad \norm{\CC{\mathcal{S}}^{k}(Tx) - \Pro_{\cap^{n}_{i=1}U_{i}}x } \leq \eta^{k} c(T) \norm{x - \Pro_{\cap^{n}_{i=1}U_{i}}x}.
	\end{align*}
\end{theorem}

\begin{proof}
	By \cref{theo:CCS:Accel:AT:x}\cref{theo:CCSAccel:AT:x:C}, we see that $0 \leq \eta \leq \frac{c(T)}{2-c(T)} \leq c(T) =\gamma^{2}   <1$. Hence, $(\forall k \in \mathbb{N})$  $\eta^{k}c(T) \leq \frac{c(T)^{k+1}}{(2 -c(T))^{k}} \leq c(T)^{k+1} =\gamma^{2(k+1)}$.
	
	By the assumption, $\Psi \subseteq \mathcal{S} \subseteq \Omega$, $\Fix T = \cap^{m}_{i=1} U_{i}$ is a closed linear subspace of $\cap_{G \in  \mathcal{S}} \Fix G$.
	Because by \cref{lem:PR:induc}, $T \in \aff \mathcal{S}$, and   $\Id \in \mathcal{S}$, it follows from \cref{defn:AccelrAT} that
	\begin{align*}
	(\forall x \in \mathcal{H}) \quad  A_{T}x \in \aff \{x, Tx\} \subseteq \aff \mathcal{S}(x).
	\end{align*} 
	Using \cref{fac:proj:commu}, we get 
	\begin{align*}
	T \Pro_{\cap^{n}_{i=1}U_{i}} =\Pro_{\cap^{n}_{i=1}U_{i}} T=\Pro_{\cap^{n}_{i=1}U_{i}}.
	\end{align*}
As we proved in \cref{theo:CCS:Accel:AT:x}, $T$ is a linear, nonexpansive and self-adjoint operator on $\mathcal{H}$. Hence, by \cref{fac:AT:Norm:MPerp}, we obtain that
	\begin{align*} 
	(\forall x \in \mathcal{H})  (\forall k \in \mathbb{N}) \quad  \bignorm{ A^{k}_{T} x - \Pro_{\cap^{n}_{i=1}U_{i}}x } 
	\leq \eta^{k} \norm{x - \Pro_{\cap^{n}_{i=1}U_{i}}x }. 
	\end{align*}
	
Therefore, the required result is obtained by applying \cref{theo:LineaConvIneq:TF}\cref{theo:LineaConvIneq:TF:W} with $W= \cap^{n}_{i=1}U_{i}$, $F=A_{T}$ and $T=\Pro_{U_{1}} \cdots \Pro_{U_{n-1}}\Pro_{U_{n}} \Pro_{U_{n-1}} \cdots \Pro_{U_{1}}$.
\end{proof}

\begin{remark} \label{rem:AT}
	Recall that in the whole section, $U_{1}, \ldots, U_{m}$  are closed linear subspaces of $\mathcal{H}$  and that the finite set $\mathcal{S}$ satisfies that $\Psi \subseteq \mathcal{S} \subseteq \Omega$. 
	Set $T :=\Pro_{U_{m}}\cdots \Pro_{U_{1}}$. By \cite[Theorem~3.7]{BDHP2003}, we know 
	\begin{align*}
	(\forall x \in \mathcal{H}) \quad A_{T}(x) = \Pro_{\aff \{x, Tx\}}(\Pro_{\cap^{m}_{i=1}U_{i} }x).
	\end{align*}
	By \cref{lem:PR:induc}, we obtain that $(\forall x \in \mathcal{H})$, $\aff \{x, Tx\} \subseteq \aff (\mathcal{S}(x))$.
	
	Moreover, $\Psi \subseteq \mathcal{S} \subseteq \Omega$ implies that $\cap^{m}_{i=1}U_{i}$ is a closed linear subspace of $\cap_{T \in \mathcal{S}} \Fix T$, so using \cref{fact:CCS:proper:NormPres:T}\cref{fact:CCS:proper:NormPres:T:AllU}, we get that
	\begin{align*}
	(\forall x \in \mathcal{H}) \quad \CC{\mathcal{S}}x= \Pro_{\aff (\mathcal{S}(x))}(\Pro_{\cap^{m}_{i=1} U_{i}} x).
	\end{align*}
	
	Hence, in some sense the $\CC{\mathcal{S}}$ can be viewed as more aggressive than the $A_{T}$ to converge to the point $\Pro_{\cap^{n}_{i=1}U_{i} }x$. Therefore, it is not surprised that the CRMs attain the linear convergence rate of the accelerated symmetric MAP in \cref{theo:CCS:Accel:AT:x,theo:CCS:Accel:AT:Tx}. 
\end{remark}

\section{Conclusion and future work}
In order to study the linear convergence of CIMs for finding the best approximation onto the intersection of fixed point sets of finitely many isometries,  we first collected and proved some properties of isometries.  
Then, we showed the linear convergence of CIMs in finite-dimensional Hilbert spaces.
Moreover, motivated by the accelerated symmetric MAP and the C-DRM, we presented two results on the linear convergence of CIMs in Hilbert spaces with first applying another operator to the initial point. 
In addition, we deduced sufficient conditions for the linear convergence of  CRMs by using the linear convergence of (symmetric) MAP and accelerated symmetric MAP. In particular, we  proved that the convergence rate of some CRMs is no worse than the sharp convergence rate of MAP
and that some CRMs attain the known linear convergence rate of the accelerated symmetric MAP. 

Let us comment on  the relation between this paper and the related literature next.

We didn't consider properties of surjective or self-adjoint isometries before.  In our previous paper \cite{BOyW2019Isometry}, we proved the circumcenter mapping induced by finite set of isometries is proper, which deduces that the CIM is well-defined and is fundamental for our study on the linear convergence of CIMs in this paper.  The linear convergence of CIMs in finite-dimensional Hilbert space are generalizations of the linear convergence of CRMs shown in \cite[Propositions~5.10 and 5.15]{BOyW2019Isometry} and \cite[Theorem~3.3]{BCS2018}.
The linear convergence of CIMs in Hilbert spaces shown in \cref{theo:LineaConvIneq:TF}\cref{theo:LineaConvIneq:TF:FFixCap} is a generalization of \cite[Theorem~1]{BCS2017} and \cite[Proposition~5.18]{BOyW2019Isometry} from reflectors to isometries, while \cref{theo:LineaConvIneq:TF}\cref{theo:LineaConvIneq:TF:W} is inspired by \cite[page~3438]{BDHP2003}. 
Note that we proved that given a linear isometry $T$, $T$ is a reflector associated with an affine subspace if and only if $T$ is self-adjoint and that generally a linear isometry is not self-adjoint, our generalizations are indeed more flexible. 
The proof of linear convergence of CRMs in Hilbert spaces by using the linear convergence of (symmetric) MAP and accelerated symmetric MAP is new.  In fact, 
compared with MAP and  DRM, some instances of those CRMs showed outstanding performance numerically but not analytically in \cite[Section~6]{BOyW2019Isometry}. Now \Cref{thm:MAP:LC,theo:CCS:Accel:AT:x} provide theoretical support for the  numerical experiments presented in \cite[Section~6]{BOyW2019Isometry}.


Let $x \in \mathcal{H}$. Let $\mathcal{S}$ be a set of finitely many isometries.  In \Cref{theo:CCS:LineaConve:F1Fn,theo:CCS:LineaConve:F1t}, we constructed operators $T_{\mathcal{S}}$ (the operators named as $A$ in \Cref{theo:CCS:LineaConve:F1Fn,theo:CCS:LineaConve:F1t}) by using the elements of $\mathcal{S}$ and proved the linear convergence of the sequence $(T_{\mathcal{S}}^{k}x)_{k \in \mathbb{N}}$ for finding $\Pro_{\cap_{T \in \mathcal{S}} \Fix T}x$ when $\mathcal{H} =\mathbb{R}^{n}$. 
Then we took advantage of the linear convergence of the sequence $(T_{\mathcal{S}}^{k}x)_{k \in \mathbb{N}}$ to prove the linear convergence of the CIM induced by the $\mathcal{S}$  in $\mathbb{R}^{n}$.  An interesting question is: \emph{can we similarly construct a $T_{\mathcal{S}}$ such that the linear convergence of $(T_{\mathcal{S}}^{k}x)_{k \in \mathbb{N}}$ implies the linear convergence of  the CIM induced by the $\mathcal{S}$  in infinite-dimensional Hilbert spaces?}
In fact, in \cref{sec:LinearConverCRMHilert}, we constructed some special sets $\mathcal{S}$ of reflectors  such that the linear convergence of (symmetric) MAP or accelerated symmetric MAP implies the linear convergence of CRMs induced by those $\mathcal{S}$. If we can answer the question above, we might be able to obtain better results than those  in \cref{sec:LinearConverCRMHilert}.

\section*{Acknowledgements}
The authors thank the anonymous referees and the editors for their
valuable comments  and suggestions.  HHB and XW were partially supported by NSERC Discovery Grants.

\addcontentsline{toc}{section}{References}

\bibliographystyle{abbrv}

\begin{thebibliography}{19}
	
	
	\bibitem{BCNPW2014}
	{\sc H.~H.~Bauschke, J.~Y.~Bello~Cruz, T.~T.~A.~Nghia, H.~M.~Phan, and
		X.~Wang}: {\em The rate of linear convergence of the {D}ouglas-{R}achford
		algorithm for subspaces is the cosine of the {F}riedrichs angle},
Journal of Approximation Theory~185, pp.~63--79, 2014.
	
	

	\bibitem{BBL1997}
	{\sc H.~H.~Bauschke, J.~M.~Borwein, and A.~S.~Lewis}: {\em The method of cyclic projections for closed convex sets in {H}ilbert space}, 
	Recent developments in optimization theory and nonlinear analysis (Jerusalem 1995), Contemporary Mathematics~204, pp.~1--38, 1997.
	
	\bibitem{BC2017}
	{\sc H.~H.~Bauschke and P.~L.~Combettes}: {\em Convex Analysis and Monotone
		Operator Theory in Hilbert Spaces}, second edition, Springer, 2017.
	
	\bibitem{BDHP2003}
	{\sc H.~H.~Bauschke, F.~Deutsch, H.~Hundal and S.~H.~Park}:
	{\em Accelerating the convergence of the method of alternating
		projections},
	Transactions of the American Mathematical Society~355, pp.~3433--3461, 2003.
	
	\bibitem{BOyW2018}
	{\sc H.~H.~Bauschke, H.~Ouyang, and X.~Wang}: {\em On circumcenters of finite
		sets in Hilbert spaces}, Linear and Nonlinear Analysis~4, pp.~271--295, 2018.
	
	\bibitem{BOyW2018Proper}
	{\sc H.~H.~Bauschke, H.~Ouyang, and X.~Wang}: 
	{\em On circumcenter mappings induced by nonexpansive operators}, 
	Pure and Applied Functional Analysis, in press. 
	
	
	\bibitem{BOyW2019Isometry}
	{\sc H.~H.~Bauschke, H.~Ouyang, and X.~Wang}: 
	{\em Circumcentered methods induced by isometries}, 
	to appear in Vietnam Journal of Mathematics,
	arXiv preprint \url{https://arxiv.org/abs/1908.11576}, 2019.

	\bibitem{BCS2017}
	{\sc R.~Behling, J.~Y.~Bello~Cruz, and L.-R.~Santos}: {\em Circumcentering the
		Douglas--Rachford method}, Numerical Algorithms~78, pp.~759--776, 2018.
	
	\bibitem{BCS2018}
	{\sc R.~Behling, J.~Y.~Bello~Cruz, and L.-R.~Santos}:  {\em On the linear
		convergence of the circumcentered-reflection method},
	Operations Research Letters~46, pp.~159--162, 2018.
	
	\bibitem{BCS2019}
	{\sc R.~Behling, J.~Y.~Bello~Cruz, and L.-R.~Santos}:  {\em The Block-wise Circumcentered-Reflection Method},
	Computational Optimization and Applications, pp.~1--25, 2019.
	
	\bibitem{BCS2020ConvexFeasibility}
	{\sc R.~Behling, J.~Y.~Bello~Cruz, and L.-R.~Santos}:
	{\em On the circumcentered-reflection method for the convex feasibility problem}, 
	arXiv preprint \url{https://arxiv.org/abs/2001.01773}, 2020. 
	
	\bibitem{D2012}
	{\sc F.~Deutsch}: {\em Best Approximation in Inner Product Spaces},
	Springer, 2012.
	
	\bibitem{DHL2019}
	{\sc N.~Dizon, J.~Hogan, and S.~B.~Lindstrom}:  {\em Circumcentering reflection methods for nonconvex feasibility problems},
	arXiv preprint \url{https://arxiv.org/abs/1910.04384},  2019.
	
	\bibitem{GK1987}
	{\sc W.~B.~Gearhart and M.~Koshy}: {\em Acceleration schemes for the method of alternating projections}, Journal of Computational and Applied Mathematics~26,
	pp.~235--249, 1989.
	
	
	\bibitem{GPR1967}
	{\sc L.~G.~Gubin , B.~T.~Polyak and E.~V.~Raik}: {\em The method of projections for finding the common point of convex sets}, USSR Computational Mathematics and Mathematical Physics~7, pp.~1--24, 1967.
	
	\bibitem{Kreyszig1989}
	{\sc E.~Kreyszig}: {\em Introductory Functional Analysis with Applications}, John Wiley \& Sons, 1989.
	
	
	\bibitem{SBLL2020}
	{\sc S.~B.~Lindstrom}:  {\em Computable centering methods for spiraling algorithms and their duals, with motivations from the theory of Lyapunov functions},
	arXiv preprint \url{https://arxiv.org/abs/2001.10784},  2020.
	
	
	\bibitem{MC2000}
	{\sc C.~Meyer}: {\em Matrix Analysis and Applied Linear Algebra}, Society for
	Industrial and Applied Mathematics, 2000.
	
\bibitem{Zarantonello1971}
{\sc E.~H.~Zarantonello}: {\em Projections on convex sets in Hilbert space and spectral
	theory. I. Projections on convex sets}, in Contributions to nonlinear functional analysis, Academic Press, pp.~237--424, 1971. 

\end{thebibliography}

\end{document}